\documentclass[12pt,a4paper]{amsart}
\usepackage{amsmath,amstext,amssymb,amsthm,amsfonts}
\usepackage{xcolor}
\newcommand{\nc}{\newcommand}

\nc{\one}{\mbox{\bf 1}}
\nc{\invtensor}{\underset{\leftarrow}{\otimes}}
\nc{\const}{\operatorname{const}}

\nc{\ad}{\operatorname{ad}}
\nc{\tr}{\operatorname{tr}}

\nc{\tp}{\operatorname{top}}
\nc{\rank}{\operatorname{rank}}
\nc{\corank}{\operatorname{corank}}
\nc{\codim}{\operatorname{codim}}
\nc{\sdim}{\operatorname{sdim}}
\nc{\mult}{\operatorname{mult}}

\nc{\spn}{\operatorname{span}}
\nc{\Sym}{\operatorname{Sym}}
\nc{\sym}{\operatorname{sym}}
\nc{\id}{\operatorname{id}}
\nc{\Id}{\operatorname{Id}}
\nc{\Ree}{\operatorname{Re}}

\nc{\htt}{\operatorname{ht}}
\nc{\sch}{\operatorname{sch}}
\nc{\str}{\operatorname{str}}

\nc{\Ker}{\operatorname{Ker}}
\nc{\rker}{\operatorname{rKer}}
\nc{\im}{\operatorname{Im}}
\nc{\osp}{\mathfrak{osp}}
\nc{\sgn}{\operatorname{sgn}}
\nc{\F}{\operatorname{F}}
\nc{\Mod}{\operatorname{Mod}}
\nc{\Mat}{\operatorname{Mat}}
\nc{\Soc}{\operatorname{Soc}}
\nc{\Inj}{\operatorname{Inj}}
\nc{\Hom}{\operatorname{Hom}}
\nc{\End}{\operatorname{End}}
\nc{\supp}{\operatorname{supp}}
\nc{\Card}{\operatorname{Card}}
\nc{\Ann}{\operatorname{Ann}}
\nc{\Ind}{\operatorname{Ind}}
\nc{\Coind}{\operatorname{Coind}}
\nc{\Res}{\operatorname{Res}}
\nc{\wt}{\operatorname{wt}}
\nc{\ch}{\operatorname{ch}}
\nc{\Stab}{\operatorname{Stab}}
\nc{\Sch}{{\mathcal S}\mbox{\em ch}}
\nc{\Irr}{\operatorname{Irr}}
\nc{\Spec}{\operatorname{Spec}}
\nc{\Prim}{\operatorname{Prim}}
\nc{\Aut}{\operatorname{Aut}}
\nc{\Ext}{\operatorname{Ext}}

\nc{\Fract}{\operatorname{Fract}}
\nc{\gr}{\operatorname{gr}}
\nc{\deff}{\operatorname{def}}
\nc{\HC}{\operatorname{HC}}

\nc{\red}{\operatorname{red}}

\nc{\wdchi}{\widetilde{\chi}}
\nc{\wdH}{\widetilde{H}}
\nc{\wdN}{\widetilde{N}}
\nc{\wdM}{\widetilde{M}}
\nc{\wdO}{\widetilde{O}}
\nc{\wdR}{\widetilde{R}}
\nc{\wdS}{\widetilde{S}}
\nc{\wdV}{\widetilde{V}}

\nc{\wdC}{\widetilde{C}}

\nc{\Obj}{\operatorname{Obj}}
\nc{\Dglie}{\operatorname{{\mathcal D}glie}}
\nc{\Fin}{\operatorname{{\mathcal F}in}}

\nc{\Adm}{\operatorname{\mathcal{A}dm}}

\nc{\cZ}{\mathcal{Z}}
\nc{\Sg}{{\cS(\fg)}}
\nc{\Shg}{{\cS(\fhg)}}
\nc{\Ug}{{\cU(\fg)}}
\nc{\Uhg}{{\cU(\fhg)}}
\nc{\Sh}{{\cS(\fh)}}
\nc{\Uh}{{\cU(\fh)}}
\nc{\Uhh}{{\cU(\fhh)}}
\nc{\Zg}{{{\mathcal{Z}}(\fg)}}

\nc{\Vir}{{\mathcal{V}ir}}
\nc{\NS}{{\mathcal{N}S}}

\nc{\tZg}{{\widetilde{\mathcal Z}({\mathfrak g})}}
\nc{\Zk}{{\mathcal Z}({\mathfrak k})}

\newcommand{\CO}{\mathcal{O}}

\newcommand{\CB}{\mathcal{B}}

\nc{\Up}{{\mathcal U}({\mathfrak p})}
\nc{\Ah}{{\mathcal A}({\mathfrak h})}
\nc{\Ag}{{\mathcal A}({\mathfrak g})}
\nc{\Ap}{{\mathcal A}({\mathfrak p})}
\nc{\Zp}{{\mathcal Z}({\mathfrak p})}
\nc{\cR}{\mathcal R}
\nc{\cS}{\mathcal S}
\nc{\cT}{\mathcal{T}}
\nc{\cY}{\mathcal Y}
\nc{\cA}{\mathcal A}
\nc{\cU}{\mathcal U}
\nc{\cH}{\mathcal H}
\nc{\cM}{\mathcal M}
\nc{\cC}{\mathcal C}
\nc{\cF}{\mathcal F}
\nc{\fg}{\mathfrak g}

\nc{\fo}{\mathfrak o}

\nc{\CR}{\mathcal R}
\nc{\Cl}{\mathcal {C}\ell}

\nc{\cW}{\mathcal{W}}
\nc{\bM}{\mathbf{M}}
\nc{\bL}{\mathbf{L}}
\nc{\bN}{\mathbf{N}}

\nc{\zq}{\mathpzc q}

\nc{\fl}{\mathfrak l}
\nc{\fn}{\mathfrak n}
\nc{\fm}{\mathfrak m}
\nc{\fp}{\mathfrak p}
\nc{\fh}{\mathfrak h}
\nc{\ft}{\mathfrak t}
\nc{\fk}{\mathfrak k}
\nc{\fb}{\mathfrak b}
\nc{\fs}{\mathfrak s}
\nc{\fB}{\mathfrak B}

\nc{\vareps}{\varepsilon}
\nc{\varesp}{\varepsilon}
\nc{\veps}{\varepsilon}

\nc{\fsl}{\mathfrak{sl}}
\nc{\fgl}{\mathfrak{gl}}
\nc{\fso}{\mathfrak{so}}

\nc{\fpq}{\mathfrak{pq}}
\nc{\fq}{\mathfrak q}
\nc{\fsq}{\mathfrak{sq}}
\nc{\fpsq}{\mathfrak{psq}}

\nc{\fhg}{\hat{\fg}}
\nc{\fhn}{\hat{\fn}}
\nc{\fhh}{\hat{\fh}}
\nc{\fhb}{\hat{\fb}}
\nc{\hrho}{\hat{\rho}}

\nc{\hsl}{\hat{\fsl}}
\nc{\fpo}{\mathfrak{po}}
\nc{\dirlim}{\underset{\rightarrow}{\lim}\,}
\nc{\nen}{\newenvironment}
\nc{\ol}{\overline}
\nc{\ul}{\underline}
\nc{\ra}{\rightarrow}
\nc{\lra}{\longrightarrow}
\nc{\Lra}{\Longrightarrow}
\nc{\bo}{\bar{1}}
\nc{\Lla}{\Longleftarrow}

\nc{\Llra}{\Longleftrightarrow}

\nc{\thla}{\twoheadleftarrow}

\nc{\lang}{(}
\nc{\rang}{)}

\nc{\hra}{\hookrightarrow}

\nc{\iso}{\overset{\sim}{\lra}}

\nc{\ssubset}{\underset{\not=}{\subset}}

\nc{\vac}{|0\rang}

\nc{\Thm}[1]{Theorem~\ref{#1}}
\nc{\Prop}[1]{Proposition~\ref{#1}}
\nc{\Lem}[1]{Lemma~\ref{#1}}
\nc{\Cor}[1]{Corollary~\ref{#1}}
\nc{\Conj}[1]{Conjecture~\ref{#1}}
\nc{\Claim}[1]{Claim~\ref{#1}}
\nc{\Defn}[1]{Definition~\ref{#1}}
\nc{\Exa}[1]{Example~\ref{#1}}
\nc{\Rem}[1]{Remark~\ref{#1}}
\nc{\Note}[1]{Note~\ref{#1}}
\nc{\Quest}[1]{Question~\ref{#1}}
\nc{\Hyp}[1]{Hypoth\`ese~\ref{#1}}
\nen{thm}[1]{\label{#1}{\bf Theorem.\ } \em}{}
\nen{prop}[1]{\label{#1}{\bf Proposition.\ } \em}{}
\nen{lem}[1]{\label{#1}{\bf Lemma.\ } \em}{}
\nen{cor}[1]{\label{#1}{\bf Corollary.\ } \em}{}
\nen{conj}[1]{\label{#1}{\bf Conjecture.\ } \em}{}

\nen{claim}[1]{\label{#1}{\bf Claim.\ } \em}{}

\nen{defn}[1]{\label{#1}{\bf Definition.\ } }{}
\nen{exa}[1]{\label{#1}{\bf Example.\ } }{}
\nen{rem}[1]{\label{#1}{\em Remark.\ } }{}
\nen{note}[1]{\label{#1}{\em Note.\ } }{}
\nen{exer}[1]{\label{#1}{\em Exercise.\ } }{}
\nen{sket}[1]{\label{#1}{\em Sketch of proof.\ } }{}
\nen{quest}[1]{\label{#1}{\bf Question.\ } \em}{}

\nen{hyp}[1]{\label{#1}{\bf Hypoth\`ese.\ } \em}{}
\title[Simple bounded highest weight modules]{Simple bounded highest weight modules of basic classical Lie superalgebras}

\author{\sc Maria Gorelik}
\address
{Department of Mathematics \\ The Weizmann Institute of Science \\ Rehovot 7610001\\ Israel}
\thanks{Research of the first author was supported in part by the Minerva foundation with funding from the Federal German Ministry for Education and Research}
\email{maria.gorelik@weizmann.ac.il}

\author{\sc Dimitar Grantcharov}
\address
{Department of Mathematics               \\
University of Texas at Arlington  \\
Arlington\\ TX~76019 \\ USA}
\thanks{
Research of the second author was supported in part by Simons Collaboration Grant 358245.}
\email{grandim@uta.edu}



\begin{document}

\maketitle

\begin{abstract} We classify all simple bounded highest weight modules of a basic classical Lie superalgebra $\mathfrak g$. In particular, our classification leads to the classification of the simple weight modules with finite weight multiplicities over all classical Lie superalgebras. We also obtain some character formulas of strongly typical bounded highest weight modules of $\mathfrak g$.
\end{abstract}

\section{Introduction}
The representation theory of Lie superalgebras have been extensively studied  in the last several decades. Remarkable progress has been  made on the study of the (super)category $\mathcal O$, see for example \cite{S2} and the references therein. On the other hand, the theory of general weight  modules of Lie superalgebras is still at its beginning stage. An important advancement in this direction was made in 2000 in \cite{DMP} where the classification of the simple weight modules with finite weight multiplicities over classical Lie superalgebras was reduced to the classification of the so-called simple cuspidal modules. This result is the superanalog of the Fernando-Futorny parabolic induction theorem for Lie algebras. The classification of the simple cuspidal modules over reductive finite-dimensional simple Lie algebras was completed by Mathieu, \cite{M}, following works of Benkart, Britten, Fernando, Futorny, Lemire, Joseph, and others, \cite{BBL}, \cite{BL}, \cite{F}, \cite{Fu}, \cite{Jo}. One important result in \cite{M} is that every simple cuspidal module is a twisted localization of a simple bounded highest weight module, where, a \emph{bounded module} by definition is a module whose set of weight multiplicities is uniformly bounded. The maximum weight multiplicity of a bounded module is called the \emph{degree} of the module.

The presentation of the simple cuspidal modules via twisted localization of highest weight modules was extended to the case of classical Lie superalgebras  in \cite{Gr}. In this way, the classification of simple weight modules with finite weight multiplicities of a classical Lie superlagebra $\mathfrak k$ is reduced to the classification of the simple bounded highest weight modules of $\mathfrak k$. The latter modules are easily classified for Lie superalgebras of type I. For Lie superalgebras of type II a classification is obtained for  Lie supealgebras of  Q-type in \cite{GG}, for the exceptional Lie superalgebra $D(2,1,a)$ in \cite{H}, and for $\mathfrak{osp} (1| 2n)$ in \cite{FGG}.  The main goal of this paper is to complete the classification in all remaining cases. In particular, by classifying the simple bounded highest weight modules of $\mathfrak{osp} (m| 2n)$,   $m=3,4,5,6$, we complete the classification of all simple weight modules with finite weight multiplicities over all classical Lie superalgebras.

Apart from the classification of simple weight modules, the category of bounded modules is interesting on its own. We believe that the results in the present paper mark the first step towards the systematic study of this category. Note that in the case of Lie algebras, bounded modules have nice geometric realizations and an equivalence of categories of bounded modules and weight modules of algebras of twisted differential operators was established in \cite{GrS1}, \cite{GrS2}. We expect that similar geometric properties of the category of bounded modules of classical Lie superlagebras hold as well.   We also expect that, like in the Lie algebra case, the injective objects in the category of bounded modules will be obtained via twisted localization functors.

We remark that in \cite{Co}, there is a classification and explicit examples of all  simple highest weight modules of degree $1$. One should note that in this classification there is a minor gap in the proof for lower-rank cases.

Most of the new results in this paper concern the highest weight bounded modules of the orthosymplectic Lie superalgebras $\mathfrak{osp} (m| 2n)$. One should note though that the above mentioned  classification is new also for the exceptional Lie superlagebras $F(4)$ an $G(3)$.  In addition to the completion of this classification, we prove that the category of $\mathcal O$-bounded $\mathfrak{osp} (1| 2n)$-modules is semisimple for $n>1$. Last, but not least, we establish explicit character formula for  strongly typical bounded modules over all basic classical Lie superalgebras.

 A crucial part in the paper plays the notion of \emph{the nonisotropic algebra} $\fg_{ni}$ associated to a Kac-Moody superalgebra $\fg$. Most  of the criteria for boundedness are expressed in terms of the components of $\fg_{ni}$.  Also, for our classification we use distinguished sets of simple roots - simple roots that contain at most one isotropic root. One of the tools used  in the paper are Enright functors - localization type of functors introduced originally by Enright in  \cite{En} for classical Lie algebras and later generalized by  \cite{IK} for Kac-Moody superalgebras.

 Our main result is~\Thm{thmbound} which describes simple highest weight bounded modules over basic classical Lie superalgebras
in terms of the highest weights with respect to
the distinguished Borel subalgebras. For all $\fg$ except for
$\fg = \mathfrak{osp}(m|2n), m\geq 5,n\geq 2$,
 we give a simple criterion, ~\Cor{cor1}. On the other hand,
 \Thm{thms2} reduces
 the remaining case $\mathfrak{osp}(m|2n), m\geq 5,n\geq 2$ to the case
 $\mathfrak{osp}(m|4)$. In Section~\ref{strtyp} we provide character formula and an  upper bound of the degree of
 a strongly typical simple highest
weight bounded module for $\mathfrak{osp}(m|2n)$.
 In Section~\ref{deg} we obtain an upper bound of the
 degree of the simple $\mathcal O$-bounded modules for the cases
 $\mathfrak{osp}(m|2n)$ with $m=3,4$ or $n=1$.

{\em Acknowledgment.}
We are grateful to A.~Joseph, I.~Penkov, and V.~Serganova for  the helpful discussions.  We also 
acknowledge the hospitality and excellent working conditions at the Weizmann Institute of Science and the University of Texas at Arlington where parts of this work were completed.

\section{Preliminaries}
Let $\fg=\fg_0\oplus\fg_1$ be a finite-dimensional Kac-Moody  superalgebra
with a fixed non-degenerate invariant bilinear form, i.e. one of the Lie superalgebras
$$\mathfrak{sl}(m|n), m\not=n, \mathfrak{gl}(n|n), \mathfrak{osp}(m|2n), D(2,1,a),F(4), G(3).$$

We fix a triangular decomposition  $\fg_0=\fn_0^-\oplus\fh\oplus\fn_0^+$
and  consider all  compatible
triangular decompositions of $\fg$, i.e. $\fg=\fn^-\oplus\fh\oplus\fn^+$ with $\fn^+_0=\fn^+\cap\fg_0$. Recall that any two triangular
decompositions  are connected by a chain of
odd reflections, see~\cite{Sint}.
We denote by $\Delta$ the root system of $\fg$ and by
$\Delta_0$ (resp., by $\Delta_1$) the set of even (resp., odd) roots.
We denote by $\Pi_0$ the set of simple roots for $\fg_0$ ($\Pi_0$ is fixed, since $\fn_0^+$ is fixed) and by $\Sigma$
a set of a base of $\Delta$.

We say that $\fg$ is indecomposable if $\Sigma$ is connected.
An indecomposable finite-dimensional Kac-Moody superalgebra is isomorphic either
to $\fgl(n|n)$ or to a basic classical Lie superalgebra which are not
equal to $\mathfrak{psl}(n|n)$.
In all examples we will use the standard notation for root systems, see~\cite{K1}.

\subsection{Notation}\label{partord}
 We set
$$\Delta_{ni}:=\{\alpha\in\Delta|\ \|\alpha\|^2\not=0\}$$
to be the set of nonisotropic roots. For $\alpha\in\Delta_{ni}$ we introduce
$\alpha^{\vee}:=\frac{2\alpha}{(\alpha,\alpha)}$ and
 the reflection $r_{\alpha}\in GL(\fh^*)$ given by
$r_{\alpha}(\mu):=\mu-(\mu,\alpha^{\vee})\alpha$.
We denote by $W$ the Weyl group of $\Delta$ (the group generated
by the reflections $r_{\alpha}$ with $\alpha\in\Delta_{ni}$).

For a base  $\Sigma$ we denote by $\rho_{\Sigma}$ its Weyl
vector.
For $\lambda\in\fh^*$ we denote by
$L(\Sigma,\lambda)$ the corresponding simple highest weight module.
Note that $L(\Sigma,\lambda)$ is a simple highest weight module for
any base $\Sigma'$ (compatible with $\Pi_0$). In the case when $\Sigma$ is fixed, we write $\rho$ for
$\rho_{\Sigma}$ and $L(\lambda)$
for $L(\Sigma,\lambda)$.  By $M(\lambda) = M(\Sigma, \lambda)$ we denote the corresponding Verma module.

For a fixed base $\Sigma$ we consider the standard
partial order on $\fh^*$:  $\mu\geq \mu'$
if $\mu-\mu'\in\mathbb{Z}_{\geq 0}\Sigma$.

For a $\fg$-module  $N$  with a locally finite action of $\fh$ we set
$$N_{\nu}:=\{v\in N|\ \forall h\in\fh\  (h-\nu(h))v=0\},\ \
\supp(N):=\{\nu\in\fh^*|\ N_{\nu}\not=0\}$$
and say that $v$ has \emph{weight} $\nu$ if $v\in N_{\nu}$.
If all weight spaces
$N_{\nu}$ are finite-dimensional, we set
$$\ch N:=\sum_{\nu\in\fh^*} \dim N_{\nu}e^{\nu}.$$

A $\fg$-module $N$ is called a \emph{weight module} if $N = \bigoplus_{\nu \in \fh^*}N_{\nu}$ and it is
{\em bounded} if it is a weight module and there is $s>0$ such that $\dim N_{\nu}<s$ for all $\nu\in\fh^*$.

\subsection{Categories $\CO, {\CO}^{inf}$}\label{cc}
We denote by ${\CO}^{inf}(\fg)$
 the full category of $\fg$-modules with the
following properties:

(C1) $\fh$ acts diagonally;

(C2) $\fn^+_0$ acts locally nilpotently.

We denote by $\CO(\fg)$ the BGG-category which is  the full category of
$\CO^{inf}$ consisting of finitely generated modules.
Note that $\CO(\fg),{\CO}^{inf}(\fg)$ do not depend on the choice of $\Sigma$.

\subsection{Kac-Moody subalgebras}\label{subsetSigma}
Fix a nonempty subset $\Sigma'\subset\Sigma$ and denote by $\ft$
the subalgebra of $\fg$ generated by $\fg_{\pm\alpha},\alpha\in\Sigma'$.
The algebra $\ft$ is either a Kac-Moody superalgebra or $\fsl(s|s)$;  $\ft\cap\fh$ is
its Cartan subalgebra
and  $\Sigma_{\ft}:=\Sigma'$ is a base; we denote by $\Delta_{\ft}$ the corresponding
root system and by $W(\ft)$ the corresponding Weyl group. One has $\Delta_{\ft}=\Delta\cap (\mathbb{Z}\Sigma')$,
see~\cite{Kbook}, Ex. 1.2.

If $\Sigma'$ is a
connected component of $\Sigma$ we call $\ft$ a {\em component} of $\fg$.

Let $\fg'\subset\fg$ be a Kac-Moody superalgebra with a triangular decomposition
$\fg'=\fn'_-\oplus\fh'\oplus \fn'_+$; we call this a subalgebra with a {\em compatible triangular
decomposition} if $\fh'\subset\fh, \fn'_{\pm}\subset \fn^{\pm}$ and $\fh$ acts diagonally on each root space of $\fg'$.

Note that for $N\in\CO^{inf}(\fg)$ one has $\Res_{\ft}^{\fg} N\in \CO^{inf}(\ft)$. On the other hand,
$\CO(\fg)$ does not have this property in general. However, the property holds in the special case ${\ft}= \fg_0$.
For each $\lambda\in\fh^*$ we denote by $\lambda_{\ft}$ the restriction of
$\lambda$ to $\ft\cap\fh$; we
denote by $M_{\ft}(\lambda_{\ft}),L_{\ft}(\lambda_{\ft})$ the corresponding $\ft$-modules. The following lemma will be useful later.

\subsubsection{}
\begin{lem}{lemPit}
(i) The $\ft$-submodule of
$L(\lambda)$ generated by a highest weight vector of $L(\lambda)$
is isomorphic to $L_{\ft}(\lambda_{\ft})$.

(ii)  Let $\fg'\subset\fg$ be Kac-Moody superalgebras with  compatible
triangular decompositions. A cyclic $\fg'$-submodule of a bounded $\fg$-module is $\fg'$-bounded.
\end{lem}
\begin{proof}
For (i) let  $v$ be a highest weight vector of $L(\lambda)$
and $L'$ be the $\ft$-submodule generated by $v$.
Clearly, $L'$  is a quotient
of $M_{\ft}(\lambda_{\ft})$. Let $uv\in L'$ be
a $\ft$-primitive vector, i.e.
 $u\in \cU(\fn^-\cap\ft)$ is such that
 $(\ft\cap\fn^+) (uv)=0$.
Take $\alpha\in\Delta^+\setminus\Delta_{\ft}$. For each
 $\beta\in\Delta_{\ft}\cap\Delta^+$ one has
$\beta-\alpha\not\in \Delta^+$ which gives
 $[\fg_{-\beta},\fg_{\alpha}]\subset\fn^+$.
 This implies
$\fg_{\alpha}(uv)=0$ and thus $uv$ is a $\fg$-primitive vector.
Therefore $uv$ is proportional to $v$, so $L'$ is simple. This gives (i).

For (ii) let $N$ be a bounded $\fg$-module and
let $N'$ be the $\fg'$-submodule generated by a vector $v'\in N$;
we may (and will) assume that $v'$ is a weight vector.
Recall that $\fh'=\fg'\cap\fh$ is the Cartan subalgebra of $\fg'$.
Set
$$\Delta':=\{\alpha\in\Delta|\ \fg_{\alpha}\subset \fg'\}.$$
Fix $\nu'\in (\fh')^*$ such that $N'_{\nu'}\not=0$. One has
$$\dim N'_{\nu'}=\sum_{\nu\in X} \dim (N_{\nu}\cap N'),$$
where
$$X:=\{\nu\in\supp N|\ \nu|_{\fh\cap\ft}=\nu',\  N_{\nu}\cap N'\not=0\}.$$
For $\nu_1,\nu_2\in X$ one has $(\nu_1-\nu_2)\in\mathbb{Z}\Delta'$,
since $N'$ is a cyclic $\fg'$-module generated by a weight vector, and
$(\nu_1-\nu_2)|_{\fh'}=0$.  Thus
 $\nu_1=\nu_2$, so $X=\{\nu_1\}$ and
$\dim N'_{\nu'}\leq \dim N_{\nu_1}$.
\end{proof}

\subsection{Root subsystems $\Delta(N), \Delta(\lambda)$}\label{rootsys}
A subset $\Delta'\subset {\Delta}_{0}$ is called a {\em root subsystem}
if $r_{\alpha}\beta\in\Delta'$ for any $\alpha,\beta\in\Delta'$.
For a root subsystem $\Delta'$ we denote by $W(\Delta')$ the subgroup
of $W$ generated by $r_{\alpha},\alpha\in\Delta'$. We set
$(\Delta')^+:=\Delta'\cap\Delta^+$ and introduce
$$\Pi((\Delta')^+):=\{\beta\in (\Delta')^+|\ r_{\beta}((\Delta')^+\setminus\{\beta\})=
(\Delta')^+\setminus\{\beta\}\}.$$
The group $W(\Delta')$ is the Coxeter group for $\Pi((\Delta')^+)$
(see, for example,~\cite{KT98}, 2.2.8--2.2.9).

For  $N\in\CO^{inf}$ we set
$$\Delta(N):=\{\alpha\in {\Delta}_{0}|\ \exists \lambda\in\supp(N)\ \text{ s.t. }
(\lambda,\alpha^{\vee})\in \mathbb{Z}\}.
$$
If $N$ is indecomposable, then
$\Delta(N)=\{\alpha\in {\Delta}_{0}|\
(\lambda,\alpha^{\vee})\in \mathbb{Z}\ \  \forall \lambda\in\supp(N)\ \}$,
since for  $\gamma\in\Delta$ and $\alpha\in\Delta_0$
one has $(\gamma,\alpha^{\vee})\in\mathbb{Z}$.
For $\lambda\in\fh^*$ we introduce
$$\Delta(\lambda):=\Delta(L(\lambda))=\{\alpha\in {\Delta}_{0}|\
(\lambda,\alpha^{\vee})\in \mathbb{Z}\}.
$$

By~\cite{Kbook}, Lem. 3.4 for
 a simple module $L$
each root space $\fg_{\alpha}$ acts either injectively or locally
nilpotenly
on $L$.  If for each $\alpha\in\Pi_0$
the root space $\fg_{-\alpha}$ acts locally
nilpotenly on $L(\lambda)$, then $L(\lambda)$ is finite-dimensional.
If $\alpha\in\Pi_0$ is such that $\alpha\in\Sigma$ or
$\frac{\alpha}{2}\in\Sigma$,
then the root space $\fg_{-\alpha}$ acts locally
nilpotenly on $L(\lambda)$ if and only if $\alpha\in\Delta(\lambda)$ and $(\lambda,\alpha^{\vee})\geq 0$.

One readily sees that $\Delta(N)$ is a root subsystem of
${\Delta}$. We set $W(N):=W(\Delta(N))$, $W(\lambda):=
W(\Delta(\lambda))$, and  $\Pi(\lambda):=\Pi(\Delta(\lambda)^+)$.
By Thm. 4.2~\cite{DGK} (the statement and the proof are the same for superalgebras)
one has
\begin{equation}\label{ext}
\Ext^1(L(\nu), L(\nu'))\not=0\
\Longrightarrow\ (\nu'+\rho)\in W(\nu)(\nu+\rho).
\end{equation}

\subsubsection{}\label{maxorb}
It is well known that the orbit $W(\mu)\mu$ contains a unique
maximal element and that $\mu$ is the maximal element in
its orbit $W(\mu)\mu$ if and only if
$(\mu,\alpha^{\vee})\geq 0$ for each
$\alpha\in\Delta(\mu)^+$. Moreover, if $\mu$
is a maximal element in  $W(\mu)\mu$, then
$Stab_W \mu$ is generated by the reflections $r_{\alpha}$
with $\alpha\in\Pi(\mu)$ such that $(\mu,\alpha^{\vee})=0$.

\subsection{Enright functors}
The Enright functors were introduced in~\cite{En}. For Kac-Moody superalgebras
the Enright functors were defined in~\cite{IK}. We will use these functors
in the following context: let $\fp$ be a Lie superalgebra containing
an $\fsl_2$-triple $(e,f,h)$ and  $\cM_a$
be the full subcategory of $\fg$-modules $N$ with the following properties:
$h$ acts diagonally with the eigenvalues in $a+\mathbb{Z}$ and
$e$ acts locally nilpotently. The Enright functor $\cC$
is a covariant functor $\cC:\cM_a\to \cM_{-a}$.
We will use the Enright functor for $\fg$ and $\fsl_2$-triple
corresponding to $\alpha\in\Delta_0$: $f\in\fg_{-\alpha}$,
$h\in\fh$, $e\in \fg_{\alpha}$; in this case we denote this functor
by $\cC_{\alpha}$.  We retain notation of~\S \ref{subsetSigma}.
Note that for indecomposable
$N\in\CO^{inf}$ the condition $\alpha\not\in\Delta(N)$ is equivalent to
$N\in \cM_a$ for $a\not\in\mathbb{Z}$.

We will use the  the following properties of the Enright functors. For the proofs we refer the reader to ~\cite{GS}.

\subsubsection{}\begin{prop}{propen}
(i) If $a\not\in\mathbb{Z}$ then $\cC:\cM_a\iso \cM_{-a}$
is an equivalence of categories.

(ii) If $\fp\subset\fg$ is a subalgebra containing the $\mathfrak{sl}_2$-triple $(e,f,h)$, then
the Enright functors commute with the restriction functor $\Res^{\fg}_{\fp}$. Namely,
$\cC^{\fp}\circ \Res^{\fg}_{\fp}=\Res^{\fg}_{\fp}\circ \ \cC^{\fg}$, where
$\cC^{\fg},\cC^{\fp}$ are Enright functors for $\fg,\fp$, respectively.

(iii) Let $\alpha\in\Pi_0$ be such that  $\alpha\in\Sigma$ or $\alpha/2\in\Sigma$
and  let $\lambda\in\fh^*$  be such that
$\alpha\not\in\Delta(\lambda)$. Then
$\cC_{\alpha}(L(\lambda))=L(r_{\alpha}(\lambda+\rho)-\rho)$ and
$\cC_{\alpha}(L_{\fg_0}(\lambda))=L(r_{\alpha}(\lambda+\rho_0)-\rho_0)$.

(iv)
If $N\in\CO^{inf}$ has a subquotient $L(\lambda)$ and
$\alpha\not\in\Delta(N)$, then $\cC_{\alpha}(L(\lambda))$ is a subquotient of
$\cC_{\alpha}(N)$.
\end{prop}

\section{Bounded modules in the case when $\Delta=\Delta_{ni}$}
In this section $\fg$ is an indecomposable finite-dimensional Kac-Moody superalgebra
without isotropic roots, i.e. $\fg$ is either a simple Lie algebra
or $\mathfrak{osp}(1|2n)$. In this case all finite-dimensional modules
are completely reducible and
$L(\lambda)$ is finite-dimensional if and only if for each simple root
$\alpha$ one has
 $(\lambda,\alpha^{\vee})\in\mathbb{Z}_{\geq 0}$.

A finite-dimensional simple Lie algebra $\ft$ admits
infinite-dimensional  bounded modules $L(\lambda)$
 only for $\fg=\mathfrak{sl}_n,\mathfrak{sp}_{2n}$. This results is proven in by \cite{BBL} generalizing the analogous result in \cite{F} for cuspidal modules.

\subsection{Bounded modules for $\mathfrak{sp}_{2n},
\mathfrak{osp}(1|2n)$}\label{boundsp}

For $\fg=\mathfrak{sp}_{2},
\mathfrak{osp}(1|2)$  all modules
in $\CO$ are bounded, since $\dim L(\lambda)_{\mu}\leq 1$
for each $\lambda,\mu\in\fh^*$.

Consider the case $\fg=\mathfrak{sp}_{2n},\mathfrak{osp}(1|2n)$ with $n>1$.
The root system $\Delta$ is of type $C_n$ or $BC_n$ and it contains a unique
copy of the root system of type $D_n$. A module $L(\lambda)$ is
an infinite-dimensional  bounded module if and only if
$$
\Delta(\lambda)=D_n,\
(\lambda+\rho,\alpha^{\vee})>0\ \text{ for each }\alpha\in \Delta(\lambda)^+.$$

For $\mathfrak{sp}_{2n}$ this is proven in \cite{M}. For $\mathfrak{osp}(1|2n)$
this is proven in~\cite{FGG} and we give another proof in~\S\ref{thmosp12n} below.
Writing the set of simple roots for $\mathfrak{sp}_{2n}$ in the form
$\{\delta_1-\delta_2,\ldots,\delta_{n-1}-\delta_n,2\delta_n\}$
 we obtain that
the root subsystem $D_n$ has a set of simple roots
$\{\delta_1-\delta_2,\ldots,\delta_{n-1}-\delta_n,\delta_{n-1}+\delta_n\}$.
Let $\lambda\in\fh^*$, and let $\lambda+\rho=\sum_{i=1}^n y_i\delta_i$. Then
 $L(\lambda)$ is
an infinite-dimensional  bounded module if and only if
\begin{equation}\label{dede}
y_1-y_2,y_2-y_3,\ldots,y_{n-1}-y_n,y_{n-1}+y_n \in \mathbb{Z}_{>0}
\end{equation}
and, in addition, $y_n\in \mathbb{Z}+\frac{1}{2}$ for  $\mathfrak{sp}_{2n}$, while
$y_n\in \mathbb{Z}$ for  $\mathfrak{osp}(1|2n)$.
Note that  $L(\lambda)$ is finite-dimensional if and only if~(\ref{dede}) holds and, in addition, $y_n\in \mathbb{Z}_{>0}$ for  $\mathfrak{sp}_{2n}$,
$y_n\in \mathbb{Z}_{>0}+\frac{1}{2}$ for  $\mathfrak{osp}(1|2n)$.

\subsection{}\label{thmosp12n}
Here we give another proof of the above-mentioned result for $\mathfrak{osp}(1|2n)$.

\begin{thm}{thm12n}
Let $\fg=\mathfrak{osp}(1|2n), n>1$. A module $L(\lambda)$ is
an infinite-dimensional bounded  module if and only if
$\Delta(\lambda)=D_n$ and
\begin{equation}\label{lambdaD}
(\lambda+\rho,\alpha^{\vee})>0\ \text{ for each }\alpha\in \Delta(\lambda)^+.\end{equation}
\end{thm}
\begin{proof}
One has $\fg_0=\mathfrak{sp}_{2n}$ and $\fg$ admits a unique base $\Sigma$ compatible with $\Pi_0$ (since $\Delta$ does not have isotropic roots).  One has
$$\Sigma=\Pi'\cup\{\delta_n\},\ \Pi_0=\Pi'\cup\{2\delta_n\},\ \Pi(D_n)=\Pi'\cup\{\delta_{n-1}+\delta_n\},$$
where
$\Pi':=\{\delta_1-\delta_2,\ldots,\delta_{n-1}-\delta_n\}$.
Let $v$ be a highest weight vector of $L(\lambda)$.

Assume that $L(\lambda)$ is bounded.
A $\fg_0$-module generated by $v$ is a quotient
of $M_{\mathfrak{sp}_{2n}}(\lambda)$, so
$L_{\mathfrak{sp}_{2n}}(\lambda)$ is a subquotient
of $\Res_{\fg_0}^{\fg} L(\lambda)$. Therefore
$L_{\mathfrak{sp}_{2n}}(\lambda)$ is bounded.

If $L_{\mathfrak{sp}_{2n}}(\lambda)$ is finite-dimensional, then
for each $\alpha\in\Pi_0$ the root space
$\fg_{-\alpha}$ acts nilpotently on $v$ and so
$L(\lambda)$ is finite-dimensional, see~\S \ref{rootsys}.

If $L_{\mathfrak{sp}_{2n}}(\lambda)$ is an infinite-dimensional bounded module,
then, by~\S \ref{boundsp},
$\Delta(\lambda)=D_n$ and thus $\fg_{-\delta_n}v\not=0$.
Since $\fn^+_0(\fg_{-\delta_n}v)=0$, the module
$\Res_{\fg_0}^{\fg} L(\lambda)$ has a primitive vector
of weight $\lambda-\delta_n$ and thus has a subquotient
isomorphic to $L_{\mathfrak{sp}_{2n}}(\lambda-\delta_n)$.
Hence $L_{\mathfrak{sp}_{2n}}(\lambda-\delta_n)$ is bounded. Since
$\Delta(\lambda-\delta_n)=\Delta(\lambda)=D_n$, the boundedness of  $L_{\mathfrak{sp}_{2n}}(\lambda-\delta_n)$   gives
$$(\lambda+\rho,\alpha^{\vee})=(\lambda-\delta_n+\rho_0,\alpha^{\vee})>0\
\text{ for each }\alpha\in \Pi(D_n),$$
see \S \ref{boundsp}. This establishes the ``only if'' part.

Now assume that $\Delta(\lambda)=D_n$ and that~(\ref{lambdaD}) holds.
Let us show that $L(\lambda)$ is bounded, i.e. that
$M:=\Res_{\fg_0}^{\fg} L(\lambda)$ is bounded.
Since $M\in \CO(\fg_0)$, it has
a finite length.

Therefore it is enough to show that any simple
subquotient of $M$ is a bounded module.
Let $L_{\fg_0}(\mu)$ be a
subquotient of $M$. One has $\Delta(\mu)=\Delta(\lambda)=D_n$.
By~\S \ref{boundsp} it suffices to show that
$(\mu+\rho_0,\alpha)>0$ for $\alpha\in \Pi(D_n)$.
Take $\alpha\in\Pi'$.
By~(\ref{lambdaD}) the root space $\fg_{-\alpha}$
acts nilpotently on $v$ and thus locally nilpotently on $L(\lambda)$
and on $L_{\fg_0}(\mu)$.
Therefore  $(\mu+\rho_0,\alpha)>0$.
It remains to verify
that $(\mu+\rho_0,\delta_{n-1}+\delta_n)>0$.
Note that $\Delta(\lambda)=\Delta(\mu)$ does not contain $2\delta_n$.
Using~\Prop{propen}  for $\alpha=2\delta_n$
we obtain that $\cC_{2\delta_n}(L_{\fg_0}(\mu))=L_{\fg_0}(r_{\delta_n}(\mu+\rho_0)-\rho_0)$
is a subquotient of
$\cC_{2\delta_n}(L(\lambda))=L(r_{\delta_n}(\lambda+\rho)-\rho)$. Since
$\delta_{n-1}-\delta_n\in\Sigma$ and
$$(r_{\delta_n}(\lambda+\rho),\delta_{n-1}-\delta_n)=
(\lambda+\rho,\delta_{n-1}+\delta_n)\in\mathbb{Z}_{>0}$$
the root space $\fg_{\delta_n-\delta_{n-1}}$ acts locally nilpotently
on $L(r_{\delta_n}(\lambda+\rho)-\rho)$ and thus on
$L_{\fg_0}(r_{\delta_n}(\mu+\rho_0)-\rho_0)$. Hence
$$0<(r_{\delta_n}(\mu+\rho_0), \delta_{n-1}-\delta_n)=
(\mu+\rho_0, \delta_{n-1}+\delta_n)$$
as required. This completes the proof.
\end{proof}
We remark that the reasoning used to prove the boundedness of $L_{\fg_0}(\mu)$ at the end of the last proof is similar to the one used for the classification of the simple highest weight bounded modules of $\mathfrak{sp}_{2n}$, see Lemma 9.2 in \cite{M}.

\subsection{Category $\CB(\fg)$}
Retain the notation of~\S \ref{cc}.
We denote by $\CB(\fg)$ the full subcategory of ${\CO}^{inf}(\fg)$
consisting of modules $N$ such that each simple subquotient of $N$ is bounded.

If $E$ is a cyclic submodule of $N\in\CB(\fg)$, then
 $E\in \CO$, so $E$ has a finite length and thus $E$ is bounded.
As a result,  $\CB(\fg)$ is the full subcategory of ${\CO}^{inf}(\fg)$
consisting of modules $N$ such that each  cyclic submodule of $N$ is bounded.
Using~\Lem{lemPit}(ii), we obtain the following.

\subsubsection{}
\begin{cor}{cor11}
Let $\fg'\subset\fg$ be Kac-Moody superalgebras with  compatible
triangular decompositions. If $N\in \CB(\fg)$, then
$\Res^{\fg}_{\fg'} N\in\CB(\fg')$.
\end{cor}

The following result is a particular case of a more general result
about quasi-integrable modules in~\cite{GS}.

\subsubsection{}
\begin{prop}{thmcomred}
Let $\fg = \mathfrak{osp}(1|2n)$ or  $\fg = \mathfrak{sp}_{2n}$, $n>1$.

(i) The category $\CB(\fg)$ is semisimple.

(ii) If $\fg\subset \fg''$ are Kac-Moody superalgebras with  compatible
triangular decompositions, then for each $N\in \CB(\fg)$
the module $\Res^{\fg''}_{\fg}N$ is completely reducible.

\end{prop}
\begin{proof}
Note that part (ii) follows from part (i) and~\Cor{cor11}.
One easily shows (see, for example, Lemma 1.3.1 of~\cite{GK}) that to prove (i)
 it is enough to verify that each
module in $\CB = \CB (\fg)$ has a simple submodule and that
$$\Ext^1_{\CB}(L(\mu),L(\mu'))=0$$
if $L(\mu), L(\mu')$ are bounded.
Take any $N\in\CB$ and let $M$ be a cyclic submodule of $N$.
Then $M$ lies in the category $\CO$ and thus admits a simple submodule.
Hence $N$ admits a simple submodule.
Since $\fh$ acts diagonally on the  highest weight spaces of he modules in $\CB$  one has
$$\Ext^1_{\CB}(L(\mu),L(\mu))=0.$$
Let $L(\mu), L(\mu')$ be bounded modules. Using the assumption on $\fg$ and~\S \ref{maxorb}, \ref{boundsp}, \Thm{thm12n}
and (\ref{ext}), we conclude that
 $\mu+\rho$ (resp., $\mu'+\rho$) is a unique
maximal element in $W(\mu)(\mu+\rho)$ (resp., $W(\mu')(\mu'+\rho)$).
Since $\mu\not=\mu'$, one has $\mu'\not\in W(\mu)(\mu+\rho)$
and so $\Ext^1( L(\mu),L(\mu'))=0$
by~(\ref{ext}).  This gives (i).
\end{proof}

We remark that in the cases $\fg = \mathfrak{sl}_n$  and $\fg = \mathfrak{osp}(1|2)$ the category $\mathcal B (\fg)$ is not semisimple. Indeed, take for example an extension of the trivial module $L(0)$ by $L(r_{\alpha}.0)$, $\alpha \in \Pi_0$ if $\fg = \mathfrak{sl}_n$,  and the Verma module $M(0)$ with highest weight $\lambda=0$ if $\fg = \mathfrak{osp}(1|2)$.

\section{Bounded modules}
In this section $\fg$ is an indecomposable finite-dimensional Kac-Moody superalgebra.

\subsection{Algebra $\fg_{ni}$}
Let $\widetilde{\Sigma}$ be the set of all bases (compatible with $\Pi_0$) of $\fg$.
Recall that all bases in $\widetilde{\Sigma}$ are connected by chains of odd reflections;
in particular, $\Delta_{ni}\cap \Delta^+$ does not depend on the choice of $\Sigma'\in
\widetilde{\Sigma}$. Set
$$\Pi_{ni}:=\cup_{\Sigma'\in\widetilde{\Sigma}} (\Sigma'\cap\Delta_{ni}).$$

Consider a Kac-Moody superalgebra $\fg_{ni}$ with the set of simple roots
$\Pi_{ni}$,  parity function $p:\Pi_{ni}\to \mathbb{Z}_2$
given by the restriction of $p:\Delta\to  \mathbb{Z}_2$ to $\Pi_{ni}$,
 and the Cartan matrix $a_{ij}:=(\alpha_i^{\vee},\alpha_j)$
for $\alpha_i,\alpha_j\in\Pi_{ni}$.

If $\Delta_{ni}\cap \Delta_1=\emptyset$, then $\fg_{ni}\cong [\fg_0,\fg_0]$
and we identify these algebras.
If $\Delta_{ni}\cap \Delta_1\not=\emptyset$, then
$\fg=\mathfrak{osp}(2s+1|2n)$
or $G(3)$. For $\fg=
\mathfrak{osp}(1|2n)$ one has $\fg_{ni}\cong \fg$  and we identify these algebras.
For $\fg=G(3), \mathfrak{osp}(2s+1|2n)$
with $s>0$, one has $\fg_0=\ft\times \mathfrak{sp}_{2n}$ and
$\fg_{ni}=\ft\times \mathfrak{osp}(1|2n)$, where
$\ft=G_2,\mathfrak{0}_{2s+1}$ respectively;
in these cases, $\fg_{ni}$ is not a subalgebra of $\fg$.

Using the above identifications, we have $(\fg_{ni})_0=[\fg_0,\fg_0]$ and fix $\fh\cap [\fg_0,\fg_0]$ to be the  Cartan subalgebra  of $\fg_{ni}$. The
 root system of $\fg_{ni}$ is $\Delta_{ni}$.

Observe also that, with the terminology of \S \ref{subsetSigma},  the connected components of $\fg_0$ are the even parts of the connected components of $\fg_{ni}$.

\subsubsection{Distinguished bases}
It is easy to check that each   connected component $\Pi'$ of $\Pi_{ni}$
lies in a certain base $\Sigma'\in\widetilde{\Sigma}$.
In this case the base $\Sigma'$ is  {\em  distinguished}, i.e. it
contains at most one isotropic root.
For instance, for $\mathfrak{osp}(7|4)$
one has
$\Pi_{ni}=\Pi'\coprod \Pi''$, where
$$\begin{array}{ll}
\Pi'=\{\vareps_1-\varesp_2,\vareps_2-\vareps_3,
\vareps_3\},\ & \
\Sigma'=\{\delta_1-\delta_2,\delta_2-\vareps_1,
\vareps_1-\varesp_2,\vareps_2-\vareps_3,
\vareps_3\},\\
\Pi''=\{\delta_1-\delta_2,\delta_2
\}, \ &\
\Sigma''=\{\vareps_1-\varesp_2,\vareps_2-\vareps_3,
\vareps_3-\delta_1,\delta_1-\delta_2,\delta_2
\}.\end{array}$$

\subsubsection{Remark}
For the case $A(m|n)$ there are two distinguished  bases, both
containing $\Pi_0$.  For the case $D(2,1,a)$
there are three  distinguished  bases,
each containing two connected components of $\Pi_0=\Pi_{\mathfrak{sl}_2} \times \Pi_{\mathfrak{sl}_2} \times\Pi_{\mathfrak{sl}_2}$.
For other cases the number of connected components is $1$ or $2$ and
the distinguished bases are in one-to-one correspondence with the connected components of $\Pi_{ni}$.

\subsubsection{Base $\Sigma_{\ft}$}
Let $\ft$ be a component of $\fg_{ni}$  and $\Pi(\ft)$
be the  corresponding connected component of $\Pi_{ni}$.
If $\fg\not=A(m|n)$ we denote by $\Sigma_{\ft}$
a  distinguished set of simple roots containing $\Pi(\ft)$
(it is unique for $\fg\not=D(2,1,a)$).
For $\fg=A(m|n)$ we choose one distinguished set of simple roots $\Sigma$
and set $\Sigma_{\ft}:=\Sigma$ for all components $\ft$ of $\fg_{ni}$.

Since $\Pi(\ft)\subset\Sigma_{\ft}$, $\ft$
is a subalgebra of $\fg$. For instance,  $\fg=\mathfrak{osp}(2s+1|2n)$
does not contain
$\fg_{ni}=\mathfrak{o}_{2s+1}\times \mathfrak{osp}(1|2n)$,
but contains  subalgebras isomorphic to
$\mathfrak{o}_{2s+1}$ and $\mathfrak{osp}(1|2n)$.

\subsubsection{Example}
Take  $\fg=\mathfrak{osp}(5|4)$. We have $\fg_{ni}=\mathfrak{o}_5
\times \mathfrak{osp}(1|4)$. Then
$$\Sigma_{\mathfrak{o}_5}=\{\delta_1-\delta_2,\delta_2-\vareps_1,
\vareps_1-\varesp_2,
\vareps_2\},\ \ \Sigma_{\mathfrak{osp}(1|4)}=\{\vareps_1-\varesp_2,\vareps_2-\delta_1,\delta_1-\delta_2,\delta_2
\}.$$

Recall the notation $\lambda_{\ft}$ from ~\S \ref{subsetSigma}. For
 $\lambda= x_1\vareps_1+x_2\vareps_2+
y_1\delta_1+y_2\delta_2$
we have $\lambda_{\mathfrak{o}_5}=x_1\vareps_1+x_2\vareps_2$
and $\lambda_{\mathfrak{osp}(1|4)}=y_1\delta_1+y_2\delta_2$.

\subsection{}
\begin{prop}{propni}
Let $\ft$ be a component of $\fg_{ni}$ such that $\ft_0\not=A_n$.
Let $\Sigma$ contains $\Pi(\ft)$.
If $L_{\ft}(\lambda_{\ft})$ is bounded, then
 $\Res^{\fg}_{\ft} L(\lambda)\in \CB(\ft)$.
\end{prop}
\begin{proof}
Set $M:=\Res^{\fg}_{\ft} L(\lambda)$.
Note that $M\in\CO^{inf}(\ft)$.

Let $v$ be a highest weight vector of $L(\lambda)$.
By~\Lem{lemPit}(i), the $\ft$-submodule of
$L(\lambda)$ generated by $v$ is isomorphic to
$L_{\ft}(\lambda_{\ft})$.

If $L_{\ft}(\lambda_{\ft})$
is finite-dimensional, then for each $\alpha\in \Pi(\ft)$
the root spaces $\fg_{-\alpha}=\ft_{-\alpha}$ acts nilpotently on $v$
and thus acts locally nilpotently on $M$. Then $M$ is a direct
sum of finite-dimensional simple $\ft$-modules, and thus $M\in\CB(\ft)$.

Now assume that $L_{\ft}(\lambda_{\ft})$
is infinite-dimensional.
Since this modules is bounded
and $\ft_0\not=A_n$, the algebra $\ft$ is
$\mathfrak{sp}_{2n}$ or $\mathfrak{osp}(1|2n)$
with $n>1$. Moreover, $\Delta(L_{\ft}(\lambda_{\ft}))=D_n$ and
$(\lambda_{\ft}+\rho(\ft),\alpha^{\vee})>0$ for each $\alpha\in \Pi(D_n)$, where
$$\Pi_0=\Pi'\cup\{2\delta_n\},\ \Pi(D_n)=\Pi'\cup\{\delta_{n-1}+\delta_n\},
\ \ \Pi':=\{\delta_1-\delta_2,\ldots,\delta_{n-1}-\delta_n\}.$$
Since $\ft$ is a component of $\fg_{ni}$ and $\Pi(\ft)\subset\Sigma$ one has
$$\rho_{\ft}=\rho(\ft),\ \ (\lambda+\rho,\alpha)=((\lambda+\rho)_{\ft},\alpha)
=(\lambda_{\ft}+\rho_{\ft},\alpha)
\ \text{ for }\alpha\in\Delta(\ft).$$
Therefore
\begin{equation}\label{snova}
(\lambda+\rho,\alpha^{\vee})>0\ \text{ for each }\alpha\in \Pi(D_n).\end{equation}

Let $L_{\ft}(\mu)$ ($\mu\in (\fh\cap\ft)^*$) be a simple
subquotient of $M$. Let us show that $L_{\ft}(\mu)$ is bounded. One has
$$\mu\in\supp (M)=\{\nu_{\ft}|\ \nu\in\supp(L(\lambda))\},$$
so for $\alpha\in\Delta(\ft)$ one has
$$(\mu,\alpha^{\vee})\subset (\lambda,\alpha^{\vee})+\mathbb{Z}=(\lambda_{\ft},
\alpha^{\vee})+\mathbb{Z}.$$
Therefore $\Delta(L_{\ft}(\lambda_{\ft}))=\Delta(L_{\ft}(\mu))=D_n$.
By~\S \ref{boundsp} it sufficies to show that
$(\mu+\rho(\ft),\alpha)>0$ for $\alpha\in \Pi(D_n)$.
Take $\alpha\in\Pi'$.
By~(\ref{snova}) the root space $\fg_{-\alpha}$
acts nilpotently on $v$ and thus locally nilpotently on $L(\lambda)$
and on $L_{\ft}(\mu)$. Therefore  $(\mu+\rho(\ft),\alpha)>0$.
By above, $\Delta(L_{\ft}(\mu)), \Delta(L(\lambda))$ do not contain $2\delta_n$.
Using~\Prop{propen} for $\alpha=2\delta_n$
we obtain that $\cC_{2\delta_n}(L_{\ft}(\mu))=L_{\ft}(r_{\delta_n}(\mu+\rho(\ft))-\rho(\ft))$
is a subquotient of
$\cC_{2\delta_n}(L(\lambda))=L(r_{\delta_n}(\lambda+\rho)-\rho)$. Since
$\delta_{n-1}-\delta_n\in \Pi(\ft)\subset\Sigma$ and
$$(r_{\delta_n}(\lambda+\rho),\delta_{n-1}-\delta_n)=
(\lambda+\rho,\delta_{n-1}+\delta_n)\in\mathbb{Z}_{>0}$$
the root space $\fg_{\delta_n-\delta_{n-1}}$ acts locally nilpotently
on $L(r_{\delta_n}(\lambda+\rho)-\rho)$ and thus on
$L_{\ft}(r_{\delta_n}(\mu+\rho(\ft)-\rho(\ft))$. Therefore
$$0<(r_{\delta_n}(\mu+\rho(\ft)), \delta_{n-1}-\delta_n)=
(\mu+\rho(\ft), \delta_{n-1}+\delta_n)$$
as required. Hence  $L_{\ft}(\mu)$ is bounded.
We conclude that $M\in\CB(\ft)$ as required.
\end{proof}

 \subsection{}
Retain the notation of~\S \ref{subsetSigma}. For a simple $\fg$-module $L$ in $\CO$ and a component $\ft$ of $\fg_{ni}$, by  $\lambda^{\ft}\in\fh^*$ we denote the  highest weight of $L$
with respect to $\Sigma_{\ft}$, i.e.
$L=L(\Sigma_{\ft},\lambda^{\ft})$.

\begin{thm}{thmbound}
  Let $\fg$ be a finite-dimensional Kac-Moody superalgebra
  and let $L\in\CO$ be a simple
$\fg$-module.  The module $L$ is bounded if and only if
the module $L_{\ft}((\lambda^{\ft})_{\ft})$
is bounded for each component $\ft$
of $\fg_{ni}$.
\end{thm}

\subsubsection{Example}
We apply the theorem to the case we are mostly interested in: $\fg=\mathfrak{osp}(m|2n)$.
Take $\lambda\in\fh^*$ and write
$\lambda+\rho=\displaystyle\sum_{i=1}^s x_i\vareps_i+\displaystyle\sum_{i=1}^n y_i\delta_i$, where $s = \left \lfloor \frac{m}{2} \right \rfloor$.
 Assume that
 $x_i+y_j\not=0$ for $i,j=1,2$. Then
for any base $\Sigma'$ we have $L(\Sigma,\lambda)=L(\Sigma',\lambda')$,
where $\lambda+\rho=\lambda'+\rho'$, $\rho' = \rho_{\Sigma'}$. The Theorem above states that
$L(\lambda)$ is bounded if and only if
$L_{\ft}((\lambda^{\ft})_{\ft})$ is bounded for
each component $\ft$ of $\fg_{ni}$. We have
$\Pi(\ft)\subset \Sigma_{\ft}$, so
$$\rho_{\Sigma_{\ft}}=\rho_{\ft}.$$
Hence, for $\lambda$ as above, $L(\lambda)$ is bounded if and only if
$L_{\ft}((\lambda+\rho)_{\ft}-\rho_{\ft})$ is bounded for each $\ft$.
One has $\fg_{ni}=\mathfrak{o}_m\times \ft$,
where
$\ft=\mathfrak{sp}_{2n}$ if $m$ is even and $\ft=\mathfrak{osp}(1|2n)$
if $n$ is odd.

One has $L_{\ft}((\lambda+\rho)_{\ft}-\rho_{\ft})=
L_{\ft}(\displaystyle\sum_{i=1}^n y_i\delta_i-\rho_{\ft})$.
For $n=1$ this module is bounded. For $n>1$ the conditions on
$y_i$ are given in~\S \ref{boundsp}.

Consider the module
$L_{\mathfrak{o}_m}(\displaystyle\sum_{i=1}^s x_i\vareps_i -\rho_{\mathfrak{o}_m})$.
For $m=1,2,3,4$ this module is always bounded.
For $m>6$ this module is bounded only if it is finite-dimensional, i.e. if
$x_1-x_2,\ldots,x_{s-1}-x_s,2x_s\in\mathbb{Z}_{>0}$.
For $m=6$ we have $\mathfrak{o}_6\cong \fsl_4$ and the boundedness is reduced to the boundedness of a module over $\fsl_4$.
For $m=5$ one has $\mathfrak{o}_5\cong \mathfrak{sp}_4$ and
this module is bounded
 if and only if either $x_1-x_2,2x_2\in \mathbb{Z}_{>0}$
or $2x_1,2x_2\in  \mathbb{Z}_{>0}$, $x-1-x_2\not\in\mathbb{Z}$.

\subsection{Proof of~\Thm{thmbound}}
We start from the following useful lemma.

\subsubsection{}
\begin{lem}{lemrho1}
  (i) A simple $\fg$-module
  $L$ is bounded if and only if it has a bounded $\fg_0$-submodule.

 (ii)  If $L_{\fg_0}(\lambda-2\rho_1)$
is bounded, then $L(\lambda)$ is bounded.
\end{lem}
\begin{proof}
  Let $N$ be a $\fg_0$-module. One has
$\Ind^{\fg}_{\fg_0} N=N\otimes \Lambda\fg_1$ as $\fg_0$-modules.
  Therefore $\Ind^{\fg}_{\fg_0} N$ is bounded if $N$ is bounded
  and  the maximal weight of $\Ind^{\fg}_{\fg_0}L_{\fg_0}(\nu)$ is
  equal to $\nu+\sum_{\alpha\in\Delta_1^+}\alpha=\nu+2\rho_1$.
  In particular, $L(\nu+2\rho_1)$ is a subquotient of
  $\Ind^{\fg}_{\fg_0}L_{\fg_0}(\nu)$ and this gives (ii).

 For (i) let $N$ be a bounded $\fg_0$-submodule of a simple
 $\fg$-module $L$. Since
 $$\Hom_{\fg_0}(N,L)=\Hom_{\fg}(\Ind_{\fg_0}^{\fg}N, L),$$
 the module $L$ is  bounded.
\end{proof}

\subsubsection{}
Assume that $L$ is bounded. Let $\ft$ be a component of
$\fg_{ni}$. Let $v$ be a primitive vector of $L$ with respect
to the base $\Sigma_{\ft}$. Then $v$ has weight $\lambda^{\ft}$
and, by~\Lem{lemPit},
$\cU(\ft)v\cong L_{\ft}((\lambda^{\ft})_{\ft})$ is a bounded $\ft$-module.
This establishes the ``only if'' part.

Now assume that $L_{\ft}((\lambda^{\ft})_{\ft})$ is bounded for
each component of $\fg_{ni}$. Let us show that $L$ contains a
bounded $\fg_0$-submodule.
If $\fg=\mathfrak{osp}(1|2n)$, the assertion is tautological.
For $\fg=D(2,1,a)$ any $\fg_0$-submodule
of $L(\lambda)$ is bounded.

Consider the remaining case when
$\fg\not=\mathfrak{osp}(1|2n), D(2,1,a)$.
Let $\ft$ be the following component of $\fg_{ni}$:
for $\fg=\mathfrak{osp}(m|2n)$ let $\ft=\mathfrak{o}_m$, for $\fg=\fgl(m|n)$
with $m\leq n$ let $\ft=\fsl_n$,
and for $\fg=G(2)$ (resp., $F(4)$) let $\ft=G_2$ (resp.,
$\ft=\mathfrak{o}_7$).
Set $\Sigma:=\Sigma_{\ft}$ and $\lambda:=\lambda^{\ft}$.
Let $v_{\lambda}$ be the highest weight vector of $L$. Let us show that
$\cU(\fg_0)v_{\lambda}$ is a bounded $\fg_0$-module.
One has $\fg_{ni}=\ft\times\ft'$ and
$[\fg_0,\fg_0]=\ft\times \ft'_0$,
where $\ft'_0=A_1$ for $F(4), G_2$ and $\ft'_0=\mathfrak{sp}_{2n}$
(resp., $\ft'_0=\ft'=\fsl_m$)
for
$\fg=\mathfrak{osp}(m|2n)$ (resp., for $\mathfrak{gl}(m|n)$).
 Set
$$E:=\cU(\ft)v_{\lambda},\ E':=\cU(\ft'_0)v_{\lambda}.$$
By~\Lem{lemPit}(i) one has $E=L_{\ft}((\lambda^{\ft})_{\ft})$, so $E$ is a simple
bounded $\ft$-module. If $\fg=\mathfrak{gl}(m|n)$, then
$\Sigma$ contains $\Pi(\ft')$, so $E'=L_{\ft'}((\lambda^{\ft'})_{\ft'})$ is a simple
bounded $\ft'$-module.
If $\ft'_0\cong A_1$,
then any module in $\CO(\ft')$ is
bounded, so $E'$ is a bounded $\ft'$-module.
In the remaining case one has $\ft'_0=\mathfrak{sp}_{2n}, n>1$. Since $L_{\ft'}((\lambda^{\ft'})_{\ft'})$
is bounded,~\Prop{propni} implies that $Res_{\ft'}^{\ft} L(\lambda)\in\CB(\ft')$
and thus, by~\Prop{thmcomred}, any cyclic $\ft'_0$-submodule of $L(\lambda)$ is bounded.
We conclude that $E'$ is a bounded $\ft'_0$-module.

View $E\otimes E'$ as a $\ft\times\ft'_0$-module
by
$$gg'(e\otimes e'):=ge\otimes g'e'\ \text{ for }g\in\ft,g'\in\ft', e\in E,
e'\in E'.$$

By above, $E,E'$  are bounded. Each weight space of $E\otimes E'$ is of the form
$(E\otimes E')_{\nu}=E_{\nu_1}\otimes  E_{\nu_2}$,
so $E\otimes E'$ is a bounded $\ft\times\ft'_0$-module.

Set $N:=\cU(\fg_0)v_{\lambda}$. Since $\fg_0=[\fg_0,\fg_0]\times \cZ(\fg_0)$ one has
$$N=\cU([\fg_0,\fg_0])v_{\lambda}=\cU(\ft\times \ft'_0)v_{\lambda}.$$
The natural map $\phi: E\otimes E'\to \cU(\ft\times\ft'_0)v_{\lambda}=N$
defined by $uv_{\lambda}\otimes u'v_{\lambda}\mapsto uu'v_{\lambda}$
is a surjective homomorphism of $\ft\times\ft'_0$-modules.
Hence $N$ is a  bounded $\ft\times\ft'_0$-module and thus
$N$ is a bounded $\fg_0$-submodule
of $L$. Now~\Lem{lemrho1} completes the proof.
\qed

\subsection{}
To check the boundedness of $L_{\ft}((\lambda^{\ft})_{\ft})$ for all $\ft$ could be computational-heavy procedure. These computations could be shorten with the aid of  \Cor{cor1} and~\Thm{thms2} below.

It turns out that
 for $\fg\not=\mathfrak{osp}(m|2n)$, it is enough to
 consider only one distinguished set of simple roots.

\subsubsection{}
\begin{cor}{cor1}
(i) If all components of $\fg_{ni}$ have rank one, then $L(\lambda)$ is
bounded for any $\lambda$.

Assume that $\ft$ is a component of $\fg_{ni}$ of rank greater
  than one. Set  $\Sigma:=\Sigma_{\ft}$.

 (ii) If $\fg\not=\mathfrak{osp}(m|2n)$, then $L(\lambda)$ is bounded if
 and only if
 $L_{\fg_0}(\lambda)$ is bounded.

 (iii) If $\fg=\mathfrak{osp}(m|2n)$ with $m=2,3,4$ or $n=1$, then
 $L(\lambda)$ is bounded if
 and only if $L_{\ft}(\lambda)$ is bounded.
\end{cor}
\begin{proof}
  If $\ft'$ is a component of $\fg_{ni}$ of rank one, then any module in
  $\CO(\ft')$ is bounded and  (i)  follows from~\Thm{thmbound}.

 If $\ft$ is a unique component of $\fg_{ni}$ which has  rank greater than one, then~\Thm{thmbound} implies that $L(\lambda)$ is bounded if
 and only if $L_{\ft}(\lambda)$ is bounded.
 Note that  $\fg_{ni}$ contains more than one component of rank greater than one
  in the following cases: $\fg=\mathfrak{osp}(m|2n)$ with $m>4,n>1$
  and $A(m|n)$ with $m,n>1$; this gives (iii).
   For $\fg=A(m|n)$ one has $\Sigma_0\subset\Sigma_{\ft}$, so
  (ii) follows from~\Thm{thmbound}.
  \end{proof}

\subsection{Reduction to $n=2$}
Let $\fg=\mathfrak{osp}(m|2n)$. Take
$$\Sigma:=\Sigma_{\mathfrak{o}_m}.$$
For $n>2$ we consider the subalgebra
$$\mathfrak{osp}(m|4)\subset
\mathfrak{osp}(m|2n)$$
with the set of simple roots lying in $\Sigma$.
For instance, for $\mathfrak{osp}(2s+1|2n)$ we have
$$\Sigma=\{\delta_1-\delta_2,\ldots, \delta_{n-1}-\delta_n,\delta_n-\vareps_1,\ldots,
\vareps_{s-1}-\vareps_s,\vareps_s\}$$
and we take $\mathfrak{osp}(2s+1|4)$ to be the subalgebra with  the set of simple roots
$\{\delta_{n-1}-\delta_n,\delta_n-\vareps_1,\ldots,
\vareps_{s-1}-\vareps_s,\vareps_s\}$.

\subsubsection{}
\begin{thm}{thms2}
For $n>2$ the module $L_{\mathfrak{osp}(m|2n)}(\lambda)$ is bounded if
and only
if the modules
$L_{\mathfrak{sp}_{2n}}(\lambda_{\mathfrak{sp}_{2n}})$
and
$L_{\mathfrak{osp}(m|4)}(\lambda_{\mathfrak{osp}(m|4)})$
are bounded.
\end{thm}
\begin{proof}
Denote by $v_{\lambda}$ the highest weight vector of
$L(\lambda):=L_{\mathfrak{osp}(m|2n)}(\lambda)$
  and set
  $$E:=\cU(\mathfrak{o}_m)v_{\lambda},\ \ E':=\cU(\mathfrak{sp}_{2n})v_{\lambda},\ \
  E'':=\cU(\mathfrak{osp}(m|4))v_{\lambda},\ \
  N:=\cU(\fg_0)v_{\lambda}.$$

By~\Lem{lemPit}, $E''\cong
L_{\mathfrak{osp}(m|4)}(\lambda_{\mathfrak{osp}(m|4)})$.
Since $E'$ has the highest weight
$\lambda_{\mathfrak{sp}_{2n}}$, the module
$L_{\mathfrak{sp}_{2n}}(\lambda_{\mathfrak{sp}_{2n}})$ is a quotient
of $E'$.

If $L(\lambda)$ is bounded, then all modules $E,E',E'', N$ are
bounded by~\Lem{lemPit}(ii). This implies the ``only if'' part.

Now assume that $L_{\mathfrak{sp}_{2n}}(\lambda_{\mathfrak{sp}_{2n}})$
and
$L_{\mathfrak{osp}(m|4)}(\lambda_{\mathfrak{osp}(m|4)})$ are bounded.
By~\Lem{lemrho1}
(i) in order to show that $L_{\mathfrak{osp}(m|2n)}(\lambda)$ is
bounded
it is enough to verify $N$ is a bounded $\fg_0$-module.
Arguing as in the proof of~\Thm{thmbound}, we see that
$N$ is a quotient of $E\otimes E'$, where $E\otimes E'$
is viewed as $\fg_0$-module ($\fg_0=\mathfrak{o}_m\times \mathfrak{sp}_{2n}$)
and that the boundedness of $N$ follows from the boundedness of
$E$ and of $E'$. Since
$\mathfrak{o}_m\subset \mathfrak{osp}(m|4)$, $E$ is a cyclic
 $\mathfrak{o}_m$-submodule
of $E''\cong L_{\mathfrak{osp}(m|4)}(\lambda_{\mathfrak{osp}(m|4)})$, so
$E$ is bounded by~\Lem{lemPit}(ii). It remains to verify the
boundedness of $E'$.

Note that $E'$ is a $\mathfrak{sp}_{2n}$-module generated
by its highest weight  vector $v_{\lambda}$ which is of the weight
$$\lambda':=\lambda_{\mathfrak{sp}_{2n}}.$$

Write
$\Pi':=\{\delta_1-\delta_2,\ldots,\delta_{n-1}-\delta_n\},\
\Pi(\mathfrak{sp}_{2n})=\Pi'\cup\{2\delta_n\}$.
Consider the copy of $\mathfrak{sl}_n$ in $\fg$ with the set
of simple roots $\Pi'$ and the copy of  $\mathfrak{sp}_4$ in $\fg$
with the set of simple roots
$\{\delta_{n-1}-\delta_n,2\delta_n\}$.
By~\Lem{lemPit},
the  $\mathfrak{sl}_n$-submodule generated by $v_{\lambda}$ is
isomorphic to $L_{\mathfrak{sl}_n}(\lambda_{\mathfrak{sl}_n})$.
Note that  $\mathfrak{sl}_n\subset \mathfrak{sp}_{2n}$ and
$\lambda'_{\mathfrak{sl}_n}=\lambda_{\mathfrak{sl}_n}$.
By~\Lem{lemPit},
the  $\mathfrak{sl}_n$-submodule generated by the highest weight
vector in $L_{\mathfrak{sp}_{2n}}(\lambda_{\mathfrak{sp}_{2n}})$ is
isomorphic to $L_{\mathfrak{sl}_n}(\lambda_{\mathfrak{sl}_n})$.
Since $L_{\mathfrak{sp}_{2n}}(\lambda_{\mathfrak{sp}_{2n}})$ is
bounded, $L_{\mathfrak{sl}_n}(\lambda_{\mathfrak{sl}_n})$ is
finite-dimensional, see~\S \ref{boundsp}.

Since $E''$ is bounded and $\mathfrak{sp}_4\subset
\mathfrak{osp}(m|4)$, the
$\mathfrak{sp}_4$-submodule generated by $v_{\lambda}$ is
bounded.
We conclude that $E'$ is an $\mathfrak{sp}_{2n}$-module with
 the following properties:

$E'$ is generated by the highest weight vector $v_{\lambda'}$;

 $\cU(\mathfrak{sl}_n)v_{\lambda'}$ is a simple finite-dimensional
 $\mathfrak{sl}_n$-module;

 $\cU(\mathfrak{sp}_4)v_{\lambda'}$ is a simple bounded
$\mathfrak{sp}_4$-module.

By the description of the simple bonded highest weight modules of $\mathfrak{sp}_{2n}$ (see \S\ref{boundsp}), $E'$ is  bounded. This completes the proof.
 \end{proof}

\section{Strongly typical modules for  $\mathfrak{osp}(m|2n)$}
\label{strtyp}
In this section $\fg=\mathfrak{osp}(m|2n)$.

A weight $\lambda$ is called {\em strongly typical }
if $(\lambda+\rho,\beta)\not=0$ for each $\beta\in\Delta_1$;
the module $L(\lambda)$ is called strongly typical if $\lambda$ is
strongly typical.

\subsection{Notation}\label{ospm2nnot}
We set
$$s:=\left \lfloor \frac{m}{2} \right\rfloor;\ \ p(m):=0\ \text{ if $m$ is even, $p(m):=1$
  if $m$ is odd}.$$

 One has
 $\fg_{ni}=\mathfrak{o}_m\times\mathfrak{sp}_{2n}$ for even $m$ and
$\fg_{ni}=\mathfrak{o}_m\times\mathfrak{osp}_{1|2n}$ for odd $m$. We write for convenience  $\fg_{ni}=\mathfrak{o}_m\times\mathfrak{osp}_{p(m)|2n}$, where $\mathfrak{osp}_{1|2n}=\mathfrak{osp}(1|2n)$
and $\mathfrak{osp}_{0|2n}=\mathfrak{sp}_{2n}$.

\subsubsection{}
We will use the standard notations of~\cite{K1} for $\Delta$, in particular, $\Delta(\mathfrak{o}_m)$ lies in the span of $\{\vareps_i\}_{i=1}^s$
and $\Delta(\mathfrak{sp}_{2n})$ lies in the span of $\{\delta_i\}_{i=1}^n$.
We set
$$\fh_{\vareps}:=\fh\cap \mathfrak{o}_m,\ \ \fh_{\delta}:=\fh\cap
\mathfrak{sp}_{2n}.$$
We identify $(\fh\cap \mathfrak{o}_m)^*$ with
$\fh^*_{\vareps}:=\mbox{span} \{\vareps_i\}_{i=1}^s$  and
$(\fh\cap \mathfrak{sp}_{2n})^*=(\fh\cap \mathfrak{osp}(1|2n))^*$
with $\fh_{\delta}^*:=\mbox{span} \{\delta_i\}_{i=1}^n$. One has
$$\fh=\fh_{\vareps}\oplus \fh_{\delta},\ \ \
\fh^*=\fh_{\vareps}^*\oplus \fh_{\delta}^*.$$

For $\lambda=\sum a_i\vareps_i+\sum b_j\delta_j$ we
set $\lambda_{\vareps}:=\sum a_i\vareps_i,\ \lambda_{\delta}:=\sum b_j\delta_j$.

In this section we use the  base $\Sigma=\Sigma_{\mathfrak{o}_m}$, i.e.
$$\begin{array}{ll}
\Sigma=\{\delta_1-\delta_2,\delta_2-\delta_3,\ldots,\delta_n-\vareps_1,
\vareps_1-\varesp_2,\ldots,\vareps_{s-1}-\vareps_s,
\vareps_s\},\ \text{ for }\mathfrak{osp}(2s+1|2n),\\
\Sigma=\{\delta_1-\delta_2,\delta_2-\delta_3,\ldots,\delta_n-\vareps_1,
\vareps_1-\varesp_2,\ldots,\vareps_{s-1}-\vareps_s,
\vareps_{s-1}+\vareps_s\},\ \text{ for }
\mathfrak{osp}(2s|2n).
\end{array}$$

Set
$$\xi:=\sum_{i=1}^n\delta_i.$$
Then $\rho_1=\rho_0-\rho=\frac{m}{2}\xi$ and  $\rho_{\fg_{ni}}-\rho=s\xi$.
We set
$$R_0:=\prod_{\alpha\in\Delta_0^+}(1-e^{-\alpha}),\ \
R_1:=\prod_{\alpha\in\Delta_1^+}(1+e^{-\alpha}), \ \ R:=R_0/R_1$$
and define $R_{\mathfrak{o}_m}, R_{\mathfrak{sp}_{2n}}, R_{\fg_{ni}}$
similarly. It is clear that $R_0=R_{\mathfrak{o}_m} R_{\mathfrak{sp}_{2n}}$.

\subsubsection{}\label{stabmu}
For $\mu\in\fh^*$ we set
$$\Delta_0(\mu):=\{\alpha\in\Delta_0(\mathfrak{sp}_{2n})|\ (\mu,\alpha)=0\}.$$
It is well-known that $\Stab_{W(\mathfrak{sp}_{2n})}\mu$ is generated by $r_{\alpha}$
with $\alpha\in\Delta_0(\mu)$ (this follows from~\S \ref{maxorb}).

We consider the root system $B_n$ of $\mathfrak{o}_n$ with the set of simple roots
$\Pi(B_n)=\{\delta_1-\delta_2,\ldots,\delta_n\}$ and denote by
$\Delta^+(B_n)$ the corresponding set of positive roots. We set
$$\mathcal{C}^+:=\{\mu\in\fh^*|\
(\mu,\alpha^{\vee})\not\in\mathbb{Z}_{<0}\ \text{ for }
\alpha\in \Delta^+(B_n)\}.$$
Note that
for any $\lambda\in\fh^*$ there exists $w\in W(\lambda)\cap W(\mathfrak{sp}_{2n})$ such that $w(\lambda+\rho)\in\mathcal{C}^+$.

\subsection{}
\begin{thm}{thmmate}
  Let $\nu\in\fh^*$ be a strongly typical weight such that
  $\nu+\rho\in\mathcal{C}^+$ and
$(\nu+\rho,\alpha)\not=0$ for $\alpha=\delta_i+\delta_j$
  with $1\leq i,j\leq n$. Then for each $z\in W(\nu)\cap W(\mathfrak{sp}_{2n})$
  one has
$$Re^{\rho}\ch L(z.\nu)=R_{\fg_{ni}}e^{\rho_{\fg_{ni}}}
\ch L_{\fg_{ni}} (z(\nu+\rho)-\rho_{\fg_{ni}}).$$
\end{thm}

\subsubsection{Remark}
Since $\fg_{ni}=\mathfrak{o}_m\times \mathfrak{osp}_{p(m)|2n}$
$$\ch L_{\fg_{ni}}(\lambda+\rho-\rho_{\fg_{ni}})=
 \ch L_{ \mathfrak{osp}_{p(m)|2n}}((\lambda+\rho-\rho_{\fg_{ni}})_{\delta})
\cdot \ch L_{\mathfrak{o}_m}((\lambda+\rho-\rho_{\fg_{ni}})_{\vareps}),$$
so
\begin{equation}\label{nuinu}\begin{array}{ll}
 e^{\rho_{\fg_{ni}}-\rho}\cdot \ch L_{\fg_{ni}}(\lambda+\rho-\rho_{\fg_{ni}})&=
e^{s\xi}\cdot\ch L_{ \mathfrak{osp}_{p(m)|2n}}(\lambda_{\delta}-s\xi)
\cdot \ch L_{\mathfrak{o}_m}(\lambda_{\vareps}).\end{array}\end{equation}

\subsubsection{}
\begin{cor}{corapp}
  Let $\lambda\in\fh^*$ be such that
$(\lambda+\rho,\alpha)\not=0$ for each $\alpha\in\Delta_1$ and either
$\Delta_0(\lambda+\rho)=\emptyset$
or
$(\lambda+\rho,2\delta_i)\in\mathbb{Z}\setminus\{0\}$
for $i=1,\ldots,n$.
 Then
$$\ch L(\lambda)=e^{s\xi}
\prod_{i=1}^s\prod_{j=1}^n (1+e^{-\vareps_i-\delta_j})(1+e^{\vareps_i-\delta_j})
\cdot \ch L_{ \mathfrak{osp}_{p(m)|2n}}(\lambda_{\delta}-s\xi)
\cdot \ch L_{\mathfrak{o}_m}(\lambda_{\vareps}).$$
\end{cor}

\subsubsection{}
\begin{cor}{corbounded}
 Let $\lambda$ be a strongly typical weight. Then $L(\lambda)$ is
 bounded
 if and only if $L_1:=L_{ \mathfrak{osp}_{p(m)|2n}}(\lambda_{\delta}-s\xi)$ and $L_2:=
L_{\mathfrak{o}_m}(\lambda_{\vareps})$ are bounded modules
(over $\mathfrak{osp}_{p(m)|2n}$ and $\mathfrak{o}_m$ respectively).
Moreover, the degree of $L(\lambda)$ is at most
$2^{2sn}\deg L_1\cdot \deg L_2$.
\end{cor}
\subsection{Proofs of Corollaries}
\subsubsection{Proof of ~\Cor{corapp}}
For~\Cor{corapp} note that take $w\in W(\lambda)$ such that
$w(\lambda+\rho)\in\mathcal{C}^+$.
It is enough to verify that $\nu:=w.\lambda$ satisfies the assumptions of~\Thm{thmmate}. One has $\nu+\rho=w(\lambda+\rho)$.
Since $\lambda$ is strongly typicial, $\nu$ is strongly typical.

If $\Delta_0(\lambda+\rho)=\emptyset$, then $\Delta_0(\nu+\rho)=\emptyset$,
so $\nu$ satisfies the assumptions of~\Thm{thmmate}.
 Assume that
for each $i=1,\ldots,n$ we have
$(\lambda+\rho,2\delta_i)\in\mathbb{Z}\setminus\{0\}$. Since
 $w\delta_i=\pm \delta_j$ we have
$$(\nu+\rho,\delta_i^{\vee})=(w(\lambda+\rho),\delta_i^{\vee})
\in\mathbb{Z}\setminus\{0\}.$$
Since $\nu+\rho\in\mathcal{C}^+$, this gives
$(\nu+\rho,\delta_i^{\vee})\in\mathbb{Z}_{>0}$, so $\nu$
satisfies the assumptions of~\Thm{thmmate}.\qed

\subsubsection{Proof of~\Cor{corbounded}}
Let $L(\lambda)$ be bounded. Then
$L_2=L_{\mathfrak{o}_m}(\lambda_{\vareps})$ is bounded.
Take $\Sigma'$ which contains the set of simple roots for
$\mathfrak{osp}_{p(m)|2n}$ and denote by $\rho'$ the corresponding
Weyl vector. Then $\rho'_{\delta}=(\rho_{\fg_{ni}})_{\delta}$.
Since $L(\lambda)=L(\Sigma',\lambda')$ is bounded,
$L_{ \mathfrak{osp}_{p(m)|2n}}(\lambda'_{\delta})$ is bounded.
Since $\lambda$ is strongly typical, one
has $\lambda'+\rho'=\lambda+\rho$,
so
$$\lambda'_{\delta}=(\lambda+\rho-\rho_{\fg_{ni}})_{\delta}.$$
Thus $L_1=L_{ \mathfrak{osp}_{p(m)|2n}}(\lambda_{\delta}+\rho-\rho_{\fg_{ni}})$ is bounded.

Now let $\lambda$ be a strongly typical weight such that
 $L_1,L_2$ are bounded modules.
Since $L_1$ is bounded, the description of the simple bonded highest weight modules
in~\S\ref{boundsp} gives
$(\lambda+\rho,\delta_i)\in\frac{1}{2}\mathbb{Z}$ for $i=1,\ldots,n$.
From~\Cor{corapp} we conclude that $L(\lambda)$ is bounded and has degree
at most $2^{2sn}\deg L_1\cdot \deg L_2$.
\qed

\subsection{Central characters}
The rest of the section is devoted to the proof of~\Thm{thmmate}.

For a weight $\lambda\in\fh^*$ we define the $\fg$- and
$\fg_0$-central characters by
$$\chi_{\lambda}:
\cZ(\fg)\to\mathbb{C}\ \text{ s.t } z|_{L(\lambda)}=\chi_{\lambda}(z)Id,\ \ \
\chi^0_{\lambda}:
\cZ(\fg_0)\to\mathbb{C}\ \ \text{ s.t. }z |_{L_{\fg_0}(\lambda)}=\chi^0_{\lambda}(z)Id.$$

We next recall the notion ``perfect mate'' which was introduced in Section 8 of \cite{G1}.
A maximal ideal $\chi^0$ in $\cZ(\fg_0)$ is called  a {\em perfect mate} for
a maximal ideal  $\chi$ in $\cZ(\fg)$ if the following conditions are satisfied.

(i) For any Verma $\fg$-module annihilated by $\chi$, its $\fg_0$-submodule annihilated by a power
of $\chi^0$ is a Verma $\fg_0$-module.

(ii) Any $\fg$-module annihilated by $\chi$ has a non-zero vector annihilated by $\chi^0$.

If $\chi^0$ is a perfect mate for $\chi$, then Theorem 1.3.1 in~\cite{G2} establishes
an equivalence of the corresponding categories of $\fg$- and $\fg_0$-modules.

\subsubsection{}
\begin{lem}{propmate}
  Let $\nu\in\fh^*$ satisfies the assumptions of~\Thm{thmmate}.
   Then:

  (i) for each $j\in\mathbb{Z}_{>0}$ one has $\nu+\rho+\frac{j}{2}\xi\in\mathcal{C}^+$
and   $\Delta_0(\nu+\rho+\frac{j}{2}\xi)=\Delta_0(\nu+\rho)$;

 (ii) $\chi^0_{\nu}$ is a perfect mate for $\chi_{\nu}$.
  \end{lem}
 \begin{proof}
 Recall that $\nu+\rho\in\cC^+$ and $\Delta_0(\nu+\rho)\subset \{\delta_i-\delta_j\}_{i,j=1}^n$.

One has $(\xi,\alpha^{\vee})=2$ for $\alpha=\delta_i, \delta_i+\delta_j$ and $(\xi,(\delta_i-\delta_j)^{\vee})=0$
for $1\leq i<j\leq n$.  For $j\in\mathbb{Z}_{>0}$ this gives that $\nu+\rho+\frac{j}{2}\xi\in\mathcal{C}^+$ and
$\Delta_0(\nu+\rho+\frac{j}{2}\xi)\subset \{\delta_i-\delta_j\}_{i,j=1}^n$, which implies  (i).
For (ii) we use~\cite{G1}, Lemma 8.3.4, which asserts that
$\chi^0_{\nu}$ is a perfect mate for $\chi_{\nu}$
if the following conditions hold:

  (1)  $\Stab_W(\nu+\rho_0)\subset \Stab_W(\nu+\rho)$;

  (2) if $\Gamma\subset\Delta_1^+$ and $w\in W$ are such that
  \begin{equation}\label{wGamma}
w(\nu+\rho_0)=\nu+\rho_0-\sum_{\beta\in\Gamma} \beta,
\end{equation}
then  $\Gamma=\emptyset$.

One has  $W=W(\mathfrak{o}_m)\times
  W(\mathfrak{sp}_{2n})$, so for each $\mu\in\fh^*$
  $$\Stab_W(\mu)=\Stab_{W(\mathfrak{o}_m)}\mu\times
  \Stab_{W(\mathfrak{sp}_{2n})}\mu=
  \Stab_{W(\mathfrak{o}_m)}\mu_{\vareps}\times
  \Stab_{W(\mathfrak{sp}_{2n})}\mu_{\delta}.
  $$
  One has $(\nu+\rho)_{\vareps}=(\nu+\rho_0)_{\vareps}$, so
$$\Stab_{W(\mathfrak{o}_m)}(\nu+\rho)_{\vareps}  =\Stab_{W(\mathfrak{o}_m)}(\nu+\rho_0)_{\vareps}.$$
By~\S \ref{stabmu}, the group $\Stab_{W(\mathfrak{sp}_{2n})}\mu_{\delta}$ is
generated by $r_{\alpha},\alpha\in\Delta_0(\mu_{\delta})$, so (i) gives
$$\Stab_{W(\mathfrak{sp}_{2n})}(\nu+\rho_0)=
\Stab_{W(\mathfrak{sp}_{2n})}(\nu+\rho)$$
and condition (1) follows. Now let us verify condition (2). Take $w\in W$ and $\Gamma\subset \Delta_1^+$
such that~(\ref{wGamma}) holds. Write $w=w_1w_2$ with
$w_1\in W(\mathfrak{o}_m)$, $w_2\in W(\mathfrak{sp}_{2n})$, and
set
$$\gamma:=\sum_{\beta\in\Gamma}\beta,\ \mu:=(\nu+\rho_0)_{\delta}.$$
Then $\mu-w_2\mu=\gamma_{\delta}$ and $\gamma_{\delta}=0$
implies $\Gamma=\emptyset$. Thus it is enough to verify that
$\gamma_{\delta}=0$.

Write $\mu=:\sum_{i=1}^n b_i\delta_i$ and $w_2\mu=:\sum_{i=1}^n b'_i\delta_i$.
The assumptions on $\nu$ give
\begin{equation}\label{bibi}
b_i-b_j\not\in\mathbb{Z}_{<0}, \ b_i+b_j-m\not\in\mathbb{Z}_{\leq 0}\
\text{ for }\
1\leq i<j\leq n.\end{equation}

Note that $\gamma_{\delta}=\sum_{i=1}^n s_i\delta_i$, where
$s_i\in\{0,1,\ldots,m\}$ for each $i$, so
$$b_i-b_i'\in\{0,1,\ldots,m\}\ \text{ for }\ i=1,\ldots,n.$$
Since $W(\mathfrak{sp}_{2n})$ acts on $\{\delta_i\}_{i=1}^n$ by signed permutations,
one has $\{|b_i|\}_{i=1}^n=\{|b'_i|\}_{i=1}^n$ as multisets.
If for some $i,j$ one has
$b'_j=-b_i$, then $b_j+b_i=b_j-b'_j\in \{0,1,\ldots,m\}$,
a contradiction to~(\ref{bibi}). Therefore
$\{b_i\}_{i=1}^n=\{b'_i\}_{i=1}^n$ as multisets. Since $b_i\geq b'_i$
for each $i$, one has $b_i=b_i'$, that is $\gamma_{\delta}=0$ as required.
\end{proof}

\subsection{Proof of~\Thm{thmmate}}\label{konec}
Recall that $y.\mu:=y(\mu+\rho)-\rho$ for $w\in W,\mu\in\fh^*$;
we consider other shifted actions
 of the Weyl group $W$ on $\fh^*$ given by
$$y_\circ \mu:=y(\mu+\rho_0)-\rho_0,\ \  y_{\bullet}\mu:=y(\mu+\rho_{\fg_{ni}})-\rho_{\fg_{ni}}$$
and note that $y.\mu=y_\circ \mu=y_{\bullet}\mu$ if $y\in W_{\mathfrak{o}_m}$.

By~\Lem{propmate},
the central character of $\fg_0$-module $M_{\fg_0}(\nu)$
is a perfect mate for the central character of $\fg$-module
$M(\nu)$. This gives rise to
equivalence of categories, see~\cite{G2}, Theorem 1.3.1. The image of
$L(z.\nu)$ under this equivalence is $L_{\fg_0}(z {}_{\circ}\nu)$,
see~\cite{FGG}, \S 8.2.1. Therefore
\begin{equation}\label{have}
R_0e^{\rho_0} \ch {L}_{\fg_0}(z_{\circ}\nu)=
\sum_{y\in W} a^z_y e^{y(\nu+\rho_0)},\ \
R e^{\rho} \ch L(z.\nu)=\sum_{y\in W} a^z_y e^{y(\nu+\rho)}\end{equation}
for certain integers $a^z_y$, which are given in terms of Kazhdan-Lusztig polynomials for the Coxeter group $W(\nu+\rho_0)$ (note that $a^z_y$ are not uniquely defined
if $\Delta_0(\nu)\not=\emptyset$).

Set
$$\mu:=\nu+\rho-\rho_{\fg_{ni}}=\nu-s\xi.$$
Our goal is to show that
\begin{equation}\label{show}
R_{\fg_{ni}}e^{\rho_{\fg_{ni}}}
\ch {L}_{\fg_{ni}}(z_{\bullet}\mu)=\sum_{y\in W} a^z_y e^{y(\mu+\rho_{\fg_{ni}})}.\end{equation}

For each $y\in W$ one has $(y_{\circ}\mu)_{\delta}=y_{\circ}(\mu_{\delta})$, and
the analogous formula holds for $y_{\bullet}$.
 Since $z\in  W(\mathfrak{sp}_{2n})$,
one has
$
(z_{\circ}\nu)_{\vareps}=\nu_{\vareps}=\mu_{\vareps}=(z_{\bullet}\mu)_{\vareps}.
$
Hence we have the following identities:
\begin{equation}\label{numu}
\begin{array}{l}
\ch L_{\fg_0}(z_{\circ}\nu)=\ch L_{\mathfrak{o}_m}(\nu_{\vareps})\cdot \ch
L_{\mathfrak{sp}_{2n}}(z_{\circ}\nu_{\delta}),\\
\ch L_{\fg_{ni}}(z_{\bullet}\mu)
=\ch L_{\mathfrak{o}_m}(\nu_{\vareps})\cdot \ch
L_{\mathfrak{osp}_{p(m)|2n}}(z_{\bullet}\mu_{\delta}),\\
R_{\mathfrak{o}_m}e^{\rho_{\mathfrak{o}_m}}\cdot \ch
L_{\mathfrak{o}_m}(\nu_{\vareps})=\displaystyle
\sum_{x\in W_{\mathfrak{o}_m}}
b_x e^{x(\nu+\rho_0)_{\vareps}}=
\sum_{x\in W_{\mathfrak{o}_m}}
b_x e^{x(\nu+\rho)_{\vareps}},\\
R_{\mathfrak{sp}_{2n}}e^{\rho_{\mathfrak{sp}_{2n}}}\cdot
\ch L_{\mathfrak{sp}_{2n}}(z {}_{\circ}\nu_{\delta})=\displaystyle
\sum_{u\in W_{\mathfrak{sp}_{2n}}}c^z_u e^{u(\nu_{\delta}+\rho_{\mathfrak{sp}_{2n}})}
=\displaystyle
\sum_{u\in W_{\mathfrak{sp}_{2n}}}c^z_u e^{u(\nu+\rho_0)_{\delta}},\\
R_{\mathfrak{osp}_{p(m)|2n}}e^{\rho_{\mathfrak{osp}_{p(m)|2n}}}\cdot
L_{\mathfrak{osp}_{p(m)|2n}}(z_{\bullet}\mu_{\delta})=\displaystyle
\sum_{u\in W_{\mathfrak{sp}_{2n}}}d^z_u e^{u(\mu_{\delta}+\rho_{\mathfrak{osp}_{p(m)|2n}})},
\end{array}\end{equation}
where $b_x, c^{z}_u, d^z_u$ are certain integers. Therefore for each
$x\in W_{\mathfrak{o}_m}, u\in W_{\mathfrak{sp}_{2n}}$ we have
$$a^z_{xu}=b_x c^{z}_u.$$
Also, one has that
$\mu_{\delta}+\rho_{\mathfrak{osp}_{p(m)|2n}}=
(\mu+\rho_{\fg_{ni}})_{\delta}=(\nu+\rho)_{\delta}$,
so the last formula of~(\ref{numu}) can be rewritten as
$$R_{\mathfrak{osp}_{p(m)|2n}}e^{\rho_{\mathfrak{osp}_{p(m)|2n}}}\cdot
L_{\mathfrak{osp}_{p(m)|2n}}(z_{\bullet}\mu_{\delta})=\displaystyle
\sum_{u\in W_{\mathfrak{sp}_{2n}}}d^z_u e^{u(\nu+\rho)_{\delta}}.$$
Therefore for each
$x\in W_{\mathfrak{o}_m}, u\in W_{\mathfrak{sp}_{2n}}$ we have
$$R_{\fg_{ni}}e^{\rho_{\fg_{ni}}}
\ch {L}_{\fg_{ni}}(z_{\bullet}\mu)=\displaystyle\sum_{x\in W_{\mathfrak{o}_m}, u\in W_{\mathfrak{sp}_{2n}}} b_xd^z_u e^{xu(\nu+\rho)}.$$
Now ~(\ref{show}) reduces to the fact that we can choose $c^z_u,d^z_u$ in such a way that  $c^z_u=d^z_u$ for
 each $u\in W_{\mathfrak{sp}_{2n}}$
(note that $c^z_u,d^z_u$ are not uniquely defined if $\Delta_0(\nu+\rho)$ or
$\Delta_0(\nu+\rho_0)$ is not empty).

Consider the case when $m$ is odd. Combining~(\ref{have}) for $m=1$
and the weight  $\mu_{\delta}+\rho_{\mathfrak{osp}_{1|2n}}=(\nu+\rho)_{\delta}$
with the last formula of~(\ref{numu}) we get
\begin{equation}\label{have1}
R_{\mathfrak{sp}_{2n}}e^{\rho_{\mathfrak{sp}_{2n}}}\cdot
\ch L_{\mathfrak{sp}_{2n}}(z {}_{\circ}\mu_{\delta})=\displaystyle
\sum_{u\in W_{\mathfrak{sp}_{2n}}}d^z_u e^{u(\mu_{\delta}+\rho_{\mathfrak{sp}_{2n}})}
\, .
\end{equation}
Note that for even $m$ the formula~(\ref{have1}) coincides with
 the last formula of~(\ref{numu}). Hence~(\ref{have1}) holds for all $m$.
 Compare~(\ref{have1}) and the forth formula of~(\ref{numu}).
In the light of~\cite{KT}, Proposition 3.9, the required formulae
  $c^z_y=d^z_y$
follow from the following conditions:

(a) $\nu_{\delta}-\mu_{\delta}$ lies in the weight lattice of $\mathfrak{sp}_{2n}$;

(b) $(\nu_{\delta}+\rho_{\mathfrak{sp}_{2n}})(\alpha^{\vee}), (\mu_{\delta}+\rho_{\mathfrak{sp}_{2n}})(\alpha^{\vee})\not\in\mathbb{Z}_{<0}$
for each $\alpha\in\Delta(\mathfrak{sp}_{2n})$;

(c) $\Delta_0(\nu_{\delta}+\rho_{\mathfrak{sp}_{2n}})=
\Delta_0(\mu_{\delta}+\rho_{\mathfrak{sp}_{2n}})$.

Condition (a)  follows from
 $\nu_{\delta}-\mu_{\delta}=s\xi$.  For (b), (c) notice that
 $$\nu_{\delta}+\rho_{\mathfrak{sp}_{2n}}=(\nu+\rho)_{\delta}+\frac{m}{2}\xi;
 \ \ \mu_{\delta}+\rho_{\mathfrak{sp}_{2n}}=(\nu+\rho)_{\delta}+\frac{p(m)}{2}\xi.
$$
Using~\Lem{propmate}(i) we obtain (c) and
$\nu_{\delta}+\rho_{\mathfrak{sp}_{2n}},\mu_{\delta}+\rho_{\mathfrak{sp}_{2n}}\in\cC^+$;
one readily sees that these inclusions imply (b).
 This completes the proof.
\qed

\section{The cases $\mathfrak{osp}(m|2n)$ for $m=3,4$ or $n=1$}
\label{deg}
\Cor{corbounded} gives an upper bound for the degree of a simple strongly typical
highest weight bounded module. In this section we deduce an upper bound
on the degree of a simple
highest weight bounded module for the cases $m=3,4$ or $n=1$.

We retain notation of~\S \ref{ospm2nnot}.
Recall that
$\mathfrak{osp}_{p(m)|2n}$ stands for $\mathfrak{sp}_{2n}$
if $m$ is even and for $\mathfrak{osp}(1|2n)$ if $m$ is odd.

\subsection{}
\begin{thm}{thmm2}
  Let $\fg=\mathfrak{osp}(m|2)$ with the base
 $\Sigma_{\mathfrak{o}_m}$.
The module $L(\lambda)$ is bounded if and only if
the $\mathfrak{o}_m$-module
$L_{\mathfrak{o}_m}(\lambda_{\mathfrak{o}_m})$
is bounded.
The degree of $L(\lambda)$ is at most  $2^{2m}
\deg L_{\mathfrak{o}_m}(\lambda_{\mathfrak{o}_m})$.
\end{thm}
\begin{proof}
Assume that $L_{\mathfrak{o}_m}(\lambda_{\mathfrak{o}_m})$
is bounded. Set
$$\Lambda:=\{\nu\in\fh^*|\ \nu_{\mathfrak{o}_m}=\lambda_{\mathfrak{o}_m}\}.$$

If $\nu\in\Lambda$ is strongly typical, then, by~\Cor{corbounded},
for each $\mu$ one has
\begin{equation}\label{ospm2}
\dim L(\nu)_{\nu-\mu}\leq 2^{2m}
\deg\, L_{\mathfrak{o}_m}(\lambda_{\mathfrak{o}_m})
\end{equation}
since any simple highest weight  $\mathfrak{osp}_{p(m)|2}$-module
has degree $1$ (note that $\mathfrak{osp}_{p(m)|2}$ is isomorphic to
$\fsl_2$ for even $m$, and to
$\mathfrak{osp}_{1|2}$ for odd $m$).

Recall that
$\dim L(\nu)_{\nu-\mu}$ is equal to the rank of the Shapovalov
matrix $S_{\mu}(\nu)$, see~\cite{Sh}. The Shapovalov matrix is a $k\times k$ matrix (where
$k=\dim \cU(\fn)_{\mu}$) with entries in $\cS(\fh)$, and such that for each $\nu\in\fh^*$
the matrix $S_{\mu}(\nu)$ is a $k\times k$ scalar matrix.
Let $\Lambda_{st}$ be the  set of strongly typical weights in $\Lambda$.
Then  $\Lambda_{st}$  is Zariski dense in $\Lambda$.
By~(\ref{ospm2}), for $\nu\in\Lambda_{st}$ the rank of $S_{\mu}(\nu)$
is at most $d:=2^{2m}\deg\, L_{\mathfrak{o}_m}(\lambda_{\mathfrak{o}_m})$.
Hence the rank of $S_{\mu}(\nu)$
is at most $d$ for each $\nu\in\Lambda$.
Thus~(\ref{ospm2}) holds for each $\nu\in\Lambda$.
This completes the proof.\end{proof}

\subsection{}
\begin{thm}{thm34}
Let $\fg=\mathfrak{osp}(m|2n)$ for $m=3$ or $m=4$ with
the base
 $$\Sigma_{\mathfrak{osp}_{p(m)|2n}}=
   \{\vareps_1-\vareps_2,\vareps_2-\delta_1,\delta_1-\delta_2,
   \ldots,a\delta_n\},$$
where $a=1 \text{ for }m=3, a=2 \text{ for }m=4$.

The module $L(\lambda)$ is bounded if and only if
the $\mathfrak{osp}_{p(m)|2n}$-module
  $L_{\mathfrak{osp}_{p(m)|2n}}(\lambda_{\mathfrak{sp}_{2n}})$
  is
bounded. The degree of $L(\lambda)$ is at most $2^{2n} \deg
L_{\mathfrak{osp}_{p(m)|2n}}(\lambda_{\mathfrak{sp}_{2n}})$.
\end{thm}
\begin{proof}
  Assume that  $L_{\mathfrak{osp}_{p(m)|2n}}(\lambda_{\delta})$  is
bounded. Set
$$\Lambda:=\{\nu\in\fh^*|\ \nu_{\mathfrak{sp}_{2n}}=\lambda_{\mathfrak{sp}_{2n}}\}$$
and let $\Lambda_{st}$ be the
set of strongly typical weights in $\Lambda$.

Set $\Sigma:=\Sigma_{\mathfrak{osp}_{p(m)|2n}}$ and $\Sigma':=\Sigma_{\mathfrak{o}_m}$; denote by $\rho$ (resp., $\rho'$)
the Weyl vector for $\Sigma$ (resp., $\Sigma'$). Observe that
any simple highest weight  $\mathfrak{o}_{m}$-module
has degree $1$ (since $\mathfrak{o}_{3}\cong \fsl_2$ and
$\mathfrak{o}_{4}\cong \fsl_2\times\fsl_2$). If $\nu\in\Lambda_{st}$, then
$L(\nu)=L(\Sigma',\nu')$ with $\nu'+\rho'=\nu+\rho$ and~\Cor{corbounded}  gives

$$\dim L(\nu)_{\nu-\mu}=  \dim L(\Sigma',\nu')_{\nu-\mu}\leq 2^{2n}
\deg\, L_{\mathfrak{osp}_{p(m)|2n}}(\nu'_{\mathfrak{sp}_{2n}}-\left\lfloor  \frac{m}{2} \right\rfloor  \sum_{i=1}^n\delta_i)$$
for each $\mu$. One has
$$\nu'_{\mathfrak{sp}_{2n}}-\left\lfloor  \frac{m}{2} \right\rfloor \sum_{i=1}^n\delta_i=
\nu_{\mathfrak{sp}_{2n}}+\rho_{\mathfrak{sp}_{2n}}-\rho'_{\mathfrak{sp}_{2n}}-\left\lfloor \frac{m}{2}  \right\rfloor \sum_{i=1}^n\delta_i=
\nu_{\mathfrak{sp}_{2n}}.$$
Therefore  for   $\nu\in\Lambda_{st}$ one has
\begin{equation}\label{osp4m}
\dim L(\nu)_{\nu-\mu}\leq 2^{2n}
\deg\, L_{\mathfrak{osp}_{p(m)|2n}}(\nu_{\mathfrak{sp}_{2n}}).
\end{equation}
Since $\Lambda_{st}$ is Zariski dense in $\Lambda$, we can use again the last argument in the proof of~\Thm{thmm2}. Thus ~(\ref{osp4m}) holds for each $\nu\in\Lambda$.
\end{proof}


\end{document}